\documentclass{amsart}
\usepackage{hyperref}

\usepackage{amsrefs}

\usepackage{amssymb}
\usepackage{verbatim}

\hypersetup{
colorlinks,
citecolor=black,
filecolor=black,
linkcolor=black,
urlcolor=black
}
\allowdisplaybreaks

\numberwithin{equation}{section}
\newtheorem{theorem}{Theorem}[section]
\newtheorem{corollary}[theorem]{Corollary}
\newtheorem{lemma}[theorem]{Lemma}
\newtheorem{proposition}[theorem]{Proposition}

\theoremstyle{definition}

\newtheorem{definition}[theorem]{Definition}

\newcommand{\dl}{\delta}

\newcommand{\om}{\omega}
\newcommand{\sg}{\sigma}
\newcommand{\Om}{\Omega}

\newcommand{\ka}{\kappa}
\newcommand{\vk}{\varkappa}

\newcommand{\eps}{\varepsilon}

\newcommand{\N}{\mathbb{N}}
\newcommand{\C}{\mathbb{C}}
\newcommand{\Z}{\mathbb{Z}}
\newcommand{\R}{\mathbb{R}}
\newcommand{\T}{\mathbb{T}}

\newcommand{\E}{\mathcal{E}}
\newcommand{\Lc}{\mathcal{L}}
\newcommand{\I}{\mathfrak{I}}

\newcommand{\f}{\frac}

\newcommand{\na}{\nabla}
\newcommand{\dd}{\partial}
\newcommand{\lan}{\langle}
\newcommand{\ran}{\rangle}
\DeclareMathOperator{\sgn}{sgn}
\DeclareMathOperator{\Id}{Id}

\let\Re=\undefined\DeclareMathOperator{\Re}{Re}
\let\Im=\undefined\DeclareMathOperator{\Im}{Im}
\let\dist=\undefined\DeclareMathOperator{\dist}{dist}
\DeclareMathOperator{\Tr}{Tr} 
\newcommand{\op}{\text{\upshape op}}

\newcommand{\qtq}[1]{\quad\text{#1}\quad}

\newcommand{\widebar}{\overline}

\title[Hamiltonian formulation of CCM models]{The Hamiltonian formulation of\\continuum Calogero--Moser models}

\author[R.~Killip]{Rowan Killip}
\address{
Department of Mathematics\\
University of California\\Los Angeles\\CA 90095\\USA\\\& CEREMADE, CNRS\\Universit\'e Paris Dauphine--PSL\\ Place du Mar\'echal de Lattre de Tassigny\\ 75016 Paris\\ France}
\email{killip@ceremade.dauphine.fr}

\author[K.~Marsden]{Katie Marsden}
\address{Department of Mathematics, University of California, Los Angeles, CA 90095, USA}
\email{kmarsden@math.ucla.edu}

\author[M.~Vi\c{s}an]{Monica Vi\c{s}an}
\address{Department of Mathematics, University of California, Los Angeles, CA 90095, USA}
\email{visan@math.ucla.edu}

\begin{document}
\raggedbottom

\begin{abstract}
Recent well-posedness results have identified the Hardy space $L^2_+$ as the natural phase space for continuum Calogero--Moser models, both focusing and defocusing, on the line and on the torus.  In this paper, we introduce a symplectic form on this phase space and so are able to realize these models as Hamiltonian systems. Moreover, we demonstrate that previously identified conserved quantities are mutually commuting, reinforcing the notion that these models are completely integrable.  We further illustrate the utility of these structures by using them to give a new proof of global well-posedness in the critical space $L^2_+$, under the necessary mass restriction in the focusing case.

Our work also brings to light several unforeseen connections: (i) the threshold for well-posedness coincides with that for the nondegeneracy of the symplectic form; (ii) this threshold is connected through Carleman's inequality to the isoperimetric problem in the plane; (iii) the transition from the line to the torus gives rise to a modified dynamical equation. 
\end{abstract}

\maketitle

\tableofcontents

\section{Introduction}  

The continuum Calogero--Moser models describe the evolution of a complex-valued function $q(t,x)$, where the spatial variable $x$ ranges over the line $\R$ or over the torus $\T=\R/2\pi\Z$.  These models were originally introduced in the line setting, where they take the following form:
\begin{equation}\label{CCMR}
iq_t=-q_{xx}\pm 2iq\dd_x C_+(|q|^2). \tag{CCM${}_\R$}
\end{equation}
Here, $C_+$ denotes the Cauchy--Szeg\H{o} projection onto the Hardy space $L^2_+(\R)$; see \eqref{Hardy} and \eqref{CSP}.  The upper sign in \eqref{CCMR} corresponds to the focusing model, while the lower sign represents the defocusing case.  We will maintain this convention throughout the paper.

The defocusing model was introduced by Pelinovsky in \cite{PELINOVSKY1995401}, where it was shown to described the self-modulation of wavepackets at a fluid interface in the infinite-depth limit.  The focusing model was introduced in \cite{ABANOV_BETTELHEIM_WIEGMANN}, where it was derived as a continuum limit of the classical Calogero--Moser particle system \cite{Calogero, MOSER1975197}.

One important feature of \eqref{CCMR} is the invariance of the Hardy space $L^2_+(\R)$ under these flows.  Of course, this relies on well-posedness theory, which we will discuss below. At this moment, we simply wish to emphasize that the property of extending holomorphically to the upper half-plane is formally preserved by the equation.  Such holomorphic solutions were termed chiral in \cite{ABANOV_BETTELHEIM_WIEGMANN} and were a key consideration in the formulation of this equation.  Concretely, a gauge transformation was introduced precisely so that chiral solutions correspond to wave-forms that admit a holomorphic extension; in the original variables, chiral solutions corresponded to a complicated manifold in $L^2(\R)$, not even a vector space!
As noted in \cite{ABANOV_BETTELHEIM_WIEGMANN}, the gauge transformation leads to a non-canonical Poisson structure on $L^2(\R)$, whose corresponding symplectic structure is given by
\begin{equation}\label{eqn:omega_hash}
\om_q^{\#}(f,g)=\Re\lan f, i g\ran_{L^2}\pm\iint_{\R\times\R}\Re(\bar{q}f)(x)\Re(\bar{q}g)(y)\sgn(x-y)\,dx\,dy;
\end{equation}
see also \cite{gerard2024calogero}.  Note that our inner product is $\C$-linear in the second argument; by comparison, symplectic forms are merely bilinear over $\R$.

The consideration of chiral solutions has dominated work on this problem.  This requires using $L^2_+(\R)$ as the phase space and so we must develop an intrinsic Hamiltonian formulation of \eqref{CCMR} on this space.  This is the first problem addressed in this paper; see Theorem~\ref{T:R}.  We then turn to the same question posed on the torus and find that matters are more complicated; indeed, we find that the natural Hamiltonian leads to a different dynamical equation; see Theorem~\ref{T:T}.

The well-posedness theory of the four continuum Calogero--Moser models (i.e., focusing/defocusing, posed on $\R$ or $\T$) is now essentially complete in Hardy--Sobolev spaces and this success has relied crucially on manifestations of complete integrability.  However, without a symplectic structure it has not been possible to verify the Poisson commutativity of the conserved quantities.  This is our next achievement; see Proposition~\ref{P:beta comm}.  

Our final reason for developing a symplectic/Hamiltonian formulation for these models is to understand the Gibbs state and prove its dynamical invariance.  While the details of this will be taken up in a subsequent paper, we use our elaboration of the underlying structures to reprove global scaling-critical well-posedness via the method of commuting flows.  Our motivation here is that this technology has proven to be very effective in constructing dynamics in the Gibbs state, both in finite and infinite volume; see \cite{KMV:white,FKV:mKdV}.

Let us turn to our first task: understanding the Hamiltonian structure of \eqref{CCMR} posed on the line.  It was previously suggested \cite[p. 4011]{gerard2024calogero} that \eqref{eqn:omega_hash} restricts to a symplectic structure for the focusing \eqref{CCMR} on $L^2_+(\R)$.  We have discovered that this is not true!  The $2$-form \eqref{eqn:omega_hash} becomes degenerate at large masses.  Strikingly this degeneration happens at the same threshold as appears in the well-posedness theory of that model!  On the other hand, we are able to show that \eqref{eqn:omega_hash} does restrict to a valid symplectic structure in the defocusing case.  To formulate this as a theorem, we first need some notations.

We will refer to the observable
\begin{equation}\label{M defn}
M(q):= \int |q(x)|^2\,dx 
\end{equation}
as the \emph{mass}.  Here integration extends over $\R$ for models on the line, or over $\T$ in the periodic case.  The mass is conserved under all the flows we discuss; indeed, we will see that it generates the phase rotation symmetry.

The significance of $M(q)$ is further reinforced by the fact that it is invariant under the scaling symmetry of \eqref{CCMR},
\begin{equation}\label{scaling}
q(t,x) \mapsto q_\lambda(t,x) : = \lambda^{\frac12} q(\lambda^2 t, \lambda x),
\end{equation}
and so provides an intrinsic notion of the size of a solution.  We also introduce
\begin{align}\label{eqn:M_ast}
B_M := \{ q \in L^2_+ : M(q) < M\} \qtq{and} 
M_\ast:=\begin{cases} 2\pi &\text{in the focusing case,}\\ \infty &\text{in the defocusing case.} \end{cases}
\end{align}

It has been shown that \eqref{CCMR} is globally well-posed in $B_{M_*}$ but that blowup can occur in the focusing case for solutions with $M(q)>M_*$; see \cite{gerard2024calogero,MR4844677,killip2023scaling,KimKwon}.

\begin{theorem}[Hamiltonian formulation on $\R$]\label{T:R}
In the line setting, the form $\om$ defined in \eqref{eqn:omega_hash} restricts to a strong symplectic form on $B_{M_*}$; however, in the focusing case the form is degenerate on any ball $B_M$ with $M>M_*$.  The \eqref{CCMR} flows are generated by the Hamiltonians
\begin{align}
H(q)=\tfrac12\|q'\mp i qC_+(|q|^2)\|_{L^2}^2 .\label{eqn:H(q) R}
\end{align}
\end{theorem}

Recall that the mapping of a Hamiltonian $H$ to the vector field $X$ that describes the corresponding flow is mediated by the symplectic form: $X$ is chosen so that $dH(\cdot) = \omega( \cdot, X)$.  (We follow the sign conventions of \cite{MR2269239}.)  Nondegeneracy of $\omega$ is crucial for the unique solvability of this equation.  Saying that $\omega$ is a \emph{strong} symplectic form means that the mapping from the functional $dH$ to the vector $X$ is a Hilbert-space isomorphism; see \cite{chernoff2006properties}.

In \cite{badreddine2024global,Rana2}, Badreddine initiated the analysis of \eqref{CCMR} on the torus and demonstrated that this equation is globally well-posed in $B_{M_*}$.  While the existence of a Lax pair and of infinitely many conservation laws did play key roles, the question of a Hamiltonian formulation did not arise in her work.  Let us address this question next.

First we observe that a direct implementation of \eqref{eqn:omega_hash} is impossible on the circle because there is no ordering of the points and so no analogue of the signum function.  Physically, the structure \eqref{eqn:omega_hash} arose from a gauge transformation of the standard Hilbert-space structure on $L^2(\R)$.  This gauge transformation does not carry over to the circle setting: for a generic $q$, it does not yield a single-valued function on the cosets $x+2\pi\Z$.

To discover the symplectic structure on $\T$, we follow a different path: We note that convolving an $L^1(\R)$ function with $\frac12\sgn(x)$ creates a primitive (antiderivative).  As a torus analogue, we choose $\tilde\partial^{-1}$ defined in \eqref{delta inverse} and \eqref{delta inverse'}.  Notice that we first project away the constant term in the Fourier series so as to guarantee the existence of a primitive.  In this way, we are lead to introduce
\begin{equation}\label{eqn:omega}
\om_q(f,g)=\Re\lan f,\Om (q)g\ran,
\end{equation}
with
\begin{equation}\label{eqn:S_T}
\Om (q)g:=i(1\mp2 C_+ \Theta (q))g \qtq{and} \Theta (q)g:=iq\tilde{\dd}^{-1}\Re(\bar{q}g)
\end{equation}
as our candidate for a symplectic structure on $L^2_+(\T)$.  In fact, through \eqref{delta inverse}, these formulae give a unified presentation of the two geometries. 

We will show that \eqref{eqn:omega} is indeed a strong symplectic form on $B_{M_*}$ in the torus case and that it degenerates on any larger ball.  Moreover, it was through the analysis of this case that we uncovered a direct connection between the critical mass $M_*$ and Carleman's proof of the isoperimetric inequality; see Lemma~\ref{L:Carleman}.

In view of Theorem~\ref{T:R}, it is natural to imagine that choosing \eqref{eqn:H(q) R} as a Hamiltonian would lead to the dynamics \eqref{CCMR} on the torus.  This is \emph{not} the case!  In Section~\ref{sec:flows}, we will show that this choice of Hamiltonian leads to the dynamical equation
\begin{equation}\label{195}
i q_t = - q_{xx} \pm 2i q \partial_x C_+(|q|^2) \pm \tfrac{i}{2\pi} M(q) q_x \pm\tfrac{3}{\pi} P(q) q + \tfrac{3}{4\pi^2} M(q)^2 q,
\end{equation}
where $M(q)$ denotes the mass \eqref{M defn} and $P(q)$ denotes the momentum, that is, the generator of the translation symmetry under our symplectic structure; see \eqref{Hamiltonians on T} and \eqref{En defn} for further information.

On the one hand, the three new terms that appear in \eqref{195} relative to \eqref{CCMR} are very mild: they generate translations and phase rotations at rates that are dictated by conserved quantities.  On the other hand, well-posedness in $L^2$ is ruined!  The problem is the appearance of momentum in the rate of phase rotation; this functional is naturally defined on $H^{1/2}_+$ and does not extend continuously to $L^2_+$.
It is not possible to remove all of the additional terms in \eqref{195} by changing the Hamiltonian.  We propose \eqref{CCMT} below as the proper analogue of \eqref{CCMR} in the torus setting.

\begin{theorem}[Hamiltonian formulation on $\T$]\label{T:T}
The bilinear form $\om$ defined in \eqref{eqn:omega} defines a strong symplectic form on $B_{M_*}$.  In the focusing case the form is degenerate on any ball $B_M$ with $M>M_*$.  Moreover, under this structure, the Hamiltonians
\begin{align}\label{H on T}
H(q)&=\tfrac12\|q'\mp i qC_+(|q|^2)\|_{L^2}^2\mp \tfrac{3}{2\pi} P(q) M(q) - \tfrac1{8\pi^2}M(q)^3
\end{align}
generate the flows
\begin{equation}\label{CCMT}\tag{CCM${}_\T$}
i\dd_t q+q_{xx}\mp2iq\dd_x C_+(|q|^2)\pm \tfrac{i}\pi M(q)q_x=0.
\end{equation}
\end{theorem}

In Section~\ref{sec:flows}, we will verify two alternate expressions for the Hamiltonian \eqref{H on T}:
\begin{align}\label{eqn:H(q) T}
H(q)&=\tfrac12\|q'\mp i qC_+(|q|^2)\|_{L^2}^2\pm \tfrac{3}{4\pi} \|q\|_{L^2}^2\|q\|_{\dot{H}^{1/2}}^2-\tfrac{3}{8\pi}\|q\|_{L^2}^2\|q\|_{L^4}^4+\tfrac{1}{16\pi^2}\|q\|_{L^2}^6\notag\\
&= \tfrac12\|q\|_{\dot{H}^{1}}^2\mp\tfrac12\| C_+(|q|^2)\|_{\dot{H}^{1/2}}^2\mp\tfrac14\|q^2\|_{\dot{H}^{1/2}}^2\pm \tfrac1{2\pi}\|q\|_{L^2}^2\|q\|_{\dot{H}^{1/2}}^2\\
&\quad+\tfrac{1}{6}\|q\|_{L^6}^6-\tfrac1{4\pi}\|q\|_{L^2}^2\|q\|_{L^4}^4+\tfrac{1}{12\pi^2}\|q\|_{L^2}^6  . \notag
\end{align}

The existence of an infinite sequence of conserved quantities for the continuum Calogero Moser models has been known for some time; these are most simply written as
\begin{align}\label{En defn}
\mathcal E_n(q) := \langle q,\Lc_q^n q\rangle \qtq{where} \Lc_q :=-i\dd_x \mp qC_+\bar{q}
\end{align}
denotes the Lax operator associated with these models; see \cite{badreddine2024global,gerard2024calogero} and Section~\ref{S:3}.   Due to the presence of ever more derivatives, these functionals are not defined on the natural phase space $L^2_+$.  By contrast, their generating function
\begin{equation}\label{beta as GF}
\beta(\ka,q)=\sum_{n\geq0}(-1)^n\ka^{-(n+1)}\mathcal{E}_n(q) = \langle q, (\Lc_q + \kappa)^{-1} q\rangle
\end{equation}
has very good analytical properties on $L^2_+$; see \cite{killip2023scaling}.

In Proposition~\ref{P:beta comm}, we prove that $\beta(\kappa,q)$ and $\beta(\vk,q)$ Poisson commute throughout $B_{M_*}$ for all pairs $\kappa,\vk \geq 1$.  In this way, we also see that the quantities $\mathcal E_n(q)$ are mutually commuting.

The Poisson commutativity of the generating functions provides a key foundation for our final endeavor, which is to prove well-posedness of \eqref{CCMR} and \eqref{CCMT} by the method of commuting flows.  Well-posedness of \eqref{CCMR} was previously shown in \cite{badreddine2024global,killip2023scaling} in each geometry.  We have two goals in pursuing this new method of proof: it provides a serious test of our understanding of the Hamiltonian structure and it provides a validated framework for analyzing dynamics in the Gibbs states of these models.

The central figures in the method of commuting flows are a one-parameter family of Hamiltonians $H_\kappa$ that both approximate the true Hamiltonian and give rise to commuting flows.  Our particular choices are given in \eqref{eqn:H_ka R} and \eqref{eqn:H_ka T}.  In the line setting, this choice follows the scheme laid out in previous implementations of the method \cite{MR4304314,MR5031430,MR4726498,killip2024sharp,killip2019kdv}, which is based on a direct rearrangement of \eqref{beta as GF}; see \eqref{11:47}.  However, in the torus case, this approach fails (for the first time in our experience!).  We rely crucially on the introduction of the final term in \eqref{eqn:H_ka T} to prove convergence of the vector fields.  It is also essential to the method that this renormalization produce a commuting flow whose symplectic gradient we can compute.

The convergence of the $H_\kappa$ vector field to that of $H$ is documented in \eqref{H diff}.  Of necessity, there is a considerable loss of regularity.  Moreover, this convergence is only uniform over equicontinuous sets (reflecting the fact that we are working at critical regularity).  For both reasons, it is essential to understand the $L^2$-equicontinuity properties of orbits.  This subject is taken up in the latter part of Section~\ref{S:3}, beginning with Definition~\ref{D:equi}.  

To demonstrate equicontinuity, we must understand how the $L^2$ norm is distributed across Fourier space.  This connects directly to constant-coefficient operators such as $\Lc_0$. On the other hand, the conserved quantities of our hierarchy are adapted to the Lax operator $\Lc_q$ that differs from $\Lc_0$ by a perturbation that is of equal order (in terms of scaling).  Theorem~\ref{T:equiv} demonstrates a new equivalence of the Sobolev spaces adapted to $\Lc_0$ and $\Lc_q$ with good uniformity properties across the phase space $B_{M_*}$.  This provides a solid foundation for an efficient demonstration of equicontinuity and will doubtlessly be useful in future investigations of the CCM hierarchies.

Combining the convergence of the vector fields with such equicontinuity properties, we will prove the following in Section~\ref{sec:convergence}:

\begin{theorem}[Well-posedness]\label{T:GWP}
The flows \eqref{CCMR} and \eqref{CCMT} are each globally well-posed on $B_{M_*}$.   Moreover, if $M<M_*$ and $Q\subset B_M$ is $L^2$-equicontinuous, then the set of orbits
$$
Q_{\ast}:=\bigl\{e^{t\na_\om H}q_0:\,q_0\in Q,\, t\in \R \bigr\}
$$
is also $L^2$-equicontinuous.
\end{theorem}

The remainder of the paper is structured as follows:  In Section~\ref{S:Note}, we review basic notations.  In Section~\ref{S:3}, we describe the Sobolev spaces adapted to the Lax operator, as well as a corresponding theory of equicontinuity.  In Section~\ref{sec:construction}, we verify that \eqref{eqn:omega} does indeed define a strong symplectic form on $B_{M_*}$ in both geometries.  We also show the necessity of the mass restriction in the focusing case.  In Section~\ref{sec:flows}, we prove that the Hamiltonians \eqref{eqn:H(q) R} and \eqref{H on T} do indeed generate the flows \eqref{CCMR} and \eqref{CCMT}.  Moreover, we derive the flow associated to the generating function $\beta(\kappa,q)$ and so the flows of the full hierarchy.  We also provide Lax pair formulations for all these flows.  A short Section~\ref{sec:commute} then verifies Poisson commutativity of all the conserved quantities.  Regularized Hamiltonians $H_\kappa(q)$ are introduced in Section~\ref{sec:GWP_kappa} and the corresponding flows are shown to be globally well-posed.  By showing convergence of the $H_\kappa$ flows as $\kappa\to\infty$, we deduce the global well-posedness of \eqref{CCMR} and \eqref{CCMT} in Section~\ref{sec:convergence}. 

\subsection*{Acknowledgements}
R. K. was supported by NSF grant DMS-2452346 and the project ANR-25-CFFS-0004 ``PhysMathEDPInteg'' of the France 2030 program. M. V. was supported by NSF grant DMS-2348018.

\section{Notations}\label{S:Note}

Irrespective of the geometry, we define the Fourier transform via
\begin{align}\label{FT}
\widehat f(\xi) = \tfrac{1}{2\pi} \int e^{-i\xi x} f(x) \,dx .
\end{align}
Correspondingly, 
\begin{align}\label{IFT}
f(x) = \int \widehat f(\xi)  e^{i\xi x} \,d\xi \qtq{and} f(x) = \sum_{\xi\in\Z} \widehat f(\xi) e^{i\xi x} ,
\end{align}
on the line and circle, respectively. Similarly, 
\begin{gather}\label{Plancherel}
\| f \|_{L^2(\R)}^2 = 2\pi \| \widehat f\, \|_{L^2(\R)}^2,\qquad
	 \| f \|_{L^2(\T)}^2 = 2\pi \| \widehat f \, \|_{\ell^2(\Z)}^2 , \\
\label{Convolution}
\widehat{fg\mkern 2mu} = \widehat{f} * \widehat{g},
	\qtq{and} \widehat{f * g} = 2\pi \widehat f \; \widehat g .
\end{gather}

Irrespective of the geometry, the two Hardy spaces $L^2_\pm$ are the closed subspaces of $L^2$ defined by
\begin{equation}\label{Hardy}
L^2_\pm := \{ f \in L^2 : \widehat f(\xi) =0 \text{ whenever $\mp \xi>0$}\}.
\end{equation}
The corresponding orthogonal projections
\begin{align}\label{CSP}
\widehat{ C_+ f }(\xi)= 1_{[0,\infty)}(\xi) \widehat f(\xi) \qtq{and}  \widehat{ C_- f }(\xi)= 1_{(-\infty,0]}(\xi) \widehat f(\xi)
\end{align}
are known as the Cauchy--Szeg\H{o} projections.  The two spaces are complementary on the line, but not on the circle; indeed,
\begin{align}\label{Cid}
C_++C_-=1 \quad \text{on $\R$,\quad but} \quad C_+ f + C_- f = f + \widehat f(0) \quad \text{ on $\T$.}
\end{align}
Notice that in the latter case, $f\mapsto \widehat f(0)$ constitutes the orthogonal projection onto constants.

For $f\in L^1$, we define
\begin{equation}\label{delta inverse}
\begin{aligned}
\bigl[\tilde{\dd}^{-1}f\bigr](x)&:=\tfrac12\int_\R \sgn(x-y) f(y)\, dy  \qquad\, \text{on $\R$},\\
\bigl[\tilde{\dd}^{-1}f\bigr](x)&:=\tfrac12\int_0^{2\pi} \bigl(1-\tfrac{y}{\pi}) f(x-y)\, dy  \quad \text{on $\T$}.
\end{aligned}
\end{equation}
It is easy to see that this is a bounded operator from $L^1$ to $L^\infty$ and that it maps real-valued functions to real-valued functions.  We also observe that
\begin{equation}\label{delta inverse'}
\begin{aligned}
\bigl[\tilde{\dd}^{-1}f\bigr]'&=f  &&\text{or equivalently}  &&\widehat{\widetilde\partial^{-1}f} (\xi) = \tfrac{1}{i\xi} \widehat f (\xi)
	&&\text{on $\R$ and} \\
\bigl[\tilde{\dd}^{-1}f\bigr]'&=f- \widehat f(0) &&\text{or equivalently}  &&\widehat{\widetilde\partial^{-1}f} (n) = \tfrac{1-\delta_{n0}}{in} \widehat f (n)
	&&\text{on $\T$}.
\end{aligned}
\end{equation}
These formulations also make it clear that $\tilde{\dd}^{-1}$ defines an anti-selfadjoint operator on $L^2_+$, which is bounded on the circle but is unbounded on the line.

We define the Sobolev spaces $H^s(\R)$ and $\dot H^s(\R)$ in the usual fashion, with norms 
$$
\|f\|_{H^s}^2 = 2\pi\int (1 +|\xi|^2)^s \bigl|\widehat f(\xi)\bigr|^2\,d\xi \qtq{and}
\|f\|_{\dot H^s}^2 = 2\pi \int |\xi|^{2s} \bigl|\widehat f(\xi)\bigr|^2\,d\xi . 
$$ 
On the torus, the definitions are parallel with integrals replaced by sums.  The Hardy--Sobolev spaces $H^s_+$ are the $H^s$-completion of $H^s\cap L^2_+$.

The Hilbert--Schmidt classes of operators on the Hilbert spaces $L^2$ and $L^2_+$ will be denoted $\I_2$.

As we are working in a Hilbert space, it is natural to conflate the tangent space at each point with the Hilbert space itself.  Accordingly, the directional derivative of a $C^1$ function $F$ on $L^2_+$ in the direction $f\in L^2_+$ at the point $q\in L^2_+$ is given by
$$
\bigl[D_f F \bigr](q) := dF|_q(f):= \lim_{\eps\to0}\f{F(q+\eps f)-F(q)}{\eps},
$$
which is an $\R$-linear map of $f$.

Such directional derivatives admit Riesz representatives in the form of Wirtinger derivatives.  Concretely, we may write
\begin{align}\label{Wirtinger}
d F|_q(f)=\int f(x)\,\left.\f{\dl F}{\dl q}\right|_q(x)+\overline{f(x)}\,{\left.\f{\dl F}{\dl \bar{q}}\right|_q}(x)\,dx= \bigl\langle \bar f,\tfrac{\dl F}{\dl q}\bigr|_q \bigr\rangle +  \bigl\langle f, \tfrac{\dl F}{\dl \bar q}\bigr|_q\bigr\rangle,
\end{align}
where $\f{\dl F}{\dl\bar{q}}$ and $\f{\dl F}{\dl q}$ are the unique elements in $L^2_+$ and $L^2_-$, respectively, so that \eqref{Wirtinger} holds for all $f\in L^2_+$.  Note that for real-valued functions $F$ we have $\f{\dl F}{\dl{q}}=\overline{\f{\dl F}{\dl\bar{q}}}$ and so we may write
\begin{align}\label{1:51}
d F|_q(f)=2\Re\,\bigl\lan f,{\tfrac{\dl F}{\dl \bar{q}}\bigr|_q}\bigr\ran.
\end{align}

\section{Sobolev spaces adapted to the Lax operator}\label{S:3}

The centerpiece of this section is a strong quantitative identification of the Sobolev spaces adapted to the Lax operator
\begin{equation}\label{E:L in S2}
\Lc_q =-i\dd_x \mp qC_+\bar{q},
\end{equation}
which defines a selfadjoint operator on $L^2_+$. Concretely, we will prove the following:
 
\begin{theorem}\label{T:equiv}
If $q\in B_{M_*}$ and $\kappa\geq 1$, then $\Lc_q+\kappa$ is positive definite and boundedly invertible.  Moreover, for each $M< M_*$, there is a constant $C_M$ so that
\begin{equation}\label{E:equiv}
C_M^{-1}  \|(\Lc_0+\ka)^{s}f \|_{L^2_+} \leq \|(\Lc_q+\ka)^{s}f \|_{L^2_+}\leq C_M \|(\Lc_0+\ka)^{s}f \|_{L^2_+} 
\end{equation}
for all $-1\leq s\leq 1$, $\kappa\geq 1$, all $f\in H^s_+$, and all $q\in B_M$. 
\end{theorem}

The fact that $\Lc_q$ defines a selfadjoint operator for $q\in L^2_+(\R)$ was addressed in \cite{gerard2024calogero}, which identified the form domain, and again in \cite{killip2023scaling}, which identified the operator domain. (Both papers consider the line case; however, the arguments are easily adapted to the circle.)  Furthermore, \cite[Prop.~2.2]{killip2023scaling} provides an analogue of \eqref{E:equiv} that works only on equicontinuous sets of $q$, rather than balls.  By demanding equicontinuity, \cite{killip2023scaling} does not need to impose a mass restriction in the focusing case.  By contrast, Lemma~\ref{L:necessary} below shows that the mass restriction is essential for achieving the uniformity described in Theorem~\ref{T:equiv}.  

The fact that the constant in \eqref{E:equiv} depends only on the mass (a conserved quantity) significantly simplifies our proof of the well-posedness of the continuum Calogero--Moser models. We believe this result will likewise prove useful in many future investigations.  


\begin{lemma}[Necessity of mass bound]\label{L:necessary}
In the focusing case, there are sequences $q_n, f_n \in H^\infty_+$ so that
\begin{align}
\int |q_n|^2 \,dx = M_\ast \qtq{and}
	\frac{ \|(\Lc_0+2)^{s}f_n \|_{L^2_+} }{ \|(\Lc_{q_n}+2)^{s}f_n \|_{L^2_+} }  +  \frac{ \|(\Lc_{q_n}+2)^{s}f_n \|_{L^2_+} }{ \|(\Lc_0+2)^{s}f_n \|_{L^2_+} }  
	\to \infty
\end{align}
for each $s\in\R\setminus\{0\}$.
\end{lemma}

\begin{proof}
Through the study of travelling wave/soliton solutions to the Calogero--Moser models, the papers \cite{gerard2024calogero,Rana2} identified explicit families of $q$ for which $q$ itself is an eigenvector of the the Lax operator $\Lc_q$.   With a particular choice of parameters, we are lead to consider
$$
f_n(x)=q_n(x) = \frac{\sqrt{2n}}{{1-inx}} \qtq{and} f_n(x)=q_n(x) = \frac{\sqrt{2n+1}}{1-n(e^{ix}-1)}
$$
on the line and circle, respectively.  It is easy to verify that $\int |q_n|^2\,dx = 2\pi$ in both geometries.   Prior analyses show that
\begin{align}
 \Lc_{q_n} f_n = \lambda f_n
\end{align}
where $\lambda=0$ in the line setting and $\lambda=-1$ on the circle.  (This is also easily verified by direct computation.)
Evidently, for any $s\in\R$,
\begin{align}
	\|(\Lc_{q_n}+2)^{s}f_n \|_{L^2_+} \simeq_s 1 \quad\text{uniformly for $n\in \N$}.  
\end{align}

In order to finish the proof, it remains only to show that 
\begin{align*}
	\|(\Lc_0+2)^{s}f_n \|_{L^2_+} \to\infty \qtq{when $s>0$} \text{and}\quad \|(\Lc_0+2)^{s}f_n \|_{L^2_+} \to 0 \quad\text{when $s<0$.}
\end{align*}
This is easily done because we can compute the Fourier transform of $f_n$:
\begin{align*}
	\widehat f_n(\xi) &= \sqrt{2n+1} \, n^\xi (n+1)^{-\xi-1} 1_{\xi\geq 0}   &&\text{for all $\xi\in\Z$, on the circle, and} \\
	\widehat f_n(\xi) &= \sqrt{2/n} \, e^{-\xi/n} 1_{\xi\geq 0}   &&\text{for all $\xi\in\R$, on the line.} \qedhere
\end{align*}
\end{proof}

In this paper, the $2\pi$ threshold will be explained using an old inequality of Carleman \cite{Carleman}, which he used to prove the sharp isoperimetric inequality in the plane (and on minimal surfaces).  On a technical level, Carleman's argument is rather close to the proof of \cite[Lemma~A.1]{gerard2024calogero}; however, we think this alternate perspective is interesting for two reasons: (i) it provides new unity between the two geometries and (ii) it suggests that the continuum Calogero--Moser models may be intrinsically connected to an interesting geometric flow.

\begin{lemma}[Carleman, \cite{Carleman}]\label{L:Carleman}
For any $f \in L^1_+(\T)$, viewed as a holomorphic function in the unit disk, we have 
\begin{equation}\label{E:Carleman}
	\int_{-\pi}^\pi\int_0^1  |f(r e^{i\theta}) |^2\,r\,dr\,d\theta  \leq \tfrac{1}{4\pi} \biggl[\int_{-\pi}^\pi |f(e^{i\theta})|\,d\theta\biggr]^2.
\end{equation}
Analogously, any $f \in L^1_+(\R)$ satisfies
\begin{equation}\label{E:Carleman_R}
	\int_{-\infty}^\infty \int_0^\infty  |f(x+i y ) |^2\,dy\,dx  \leq \tfrac{1}{4\pi} \biggl[\int_{-\infty}^\infty |f(x)|\,dx\biggr]^2,
\end{equation}
when viewed as a holomorphic function in the upper half plane.
\end{lemma}

\begin{proof}
Carleman only considers \eqref{E:Carleman}.  We will illustrate his argument by using it to prove \eqref{E:Carleman_R}.
Given $f\in L^1_+(\R)$, we may perform an inner/outer factorization (see~\cite{MR2261424}).  The presence of an inner factor does not affect RHS\eqref{E:Carleman_R} and only serves to diminish LHS\eqref{E:Carleman_R}. Thus, we need only treat outer functions $f$, which may be written $f=g^2$ for some outer function $g\in L^2_+$.

By changing from $f\in L^1_+$ to $g\in L^2_+$, Carleman has made the problem amenable to Fourier analysis.  Observe first that the analytic extension of $f$ to the upper half-plane is given by
\begin{equation*}
f(x+iy) = \int_0^\infty \! e^{i\xi x - \xi y} \widehat f(\xi) \, d\xi \quad\text{for all $x\in\R$ and $y>0$.} 
\end{equation*}
Correspondingly, using \eqref{Plancherel}, we find
\begin{align}
	\text{LHS\eqref{E:Carleman_R}} &= 2\pi \int_0^\infty \!\! \int_0^\infty \bigl| \widehat f (\xi)|^2 e^{-2\xi y}\,d\xi\,dy =
		2\pi \int_0^\infty \bigl| \widehat f (\xi)|^2\tfrac{\,d\xi}{2\xi},
		\label{CarlFub}
\end{align}
into which we substitute $\widehat f = \widehat g * \widehat g$; see \eqref{Convolution}.  Using \eqref{Plancherel} also on RHS\eqref{E:Carleman_R}, we have reduced the claim \eqref{E:Carleman_R} to
\begin{equation}
	\int_0^\infty  \biggl| \int_0^\xi \widehat g(\xi-\eta) \widehat g(\eta) \, d\eta \biggr|^2\,\frac{\pi\,d\xi}{\xi}
		\leq \pi \biggl[\int_0^\infty |\widehat g(\xi)|^2\,d\xi\biggr]^2 .
\end{equation}
This follows by applying the Cauchy--Schwarz inequality in the $\eta$ integral. 
\end{proof}

To prove the isoperimetric inequality for a simply connected domain $\Omega$, Carleman applies \eqref{E:Carleman} to the derivative of the Riemann mapping from the unit disk onto $\Omega$.  The inequality \eqref{E:Carleman_R} performs the same role if one uses the upper half-plane as the model domain instead of the disk.  This is another reason for regarding \eqref{E:Carleman_R} as the proper analogue in this geometry, beyond the similarity of the proof.

The proof of Theorem~\ref{T:equiv} rests on treating $\Lc_q$ as a perturbation of $\Lc_0$. To verify that $\Lc_q+\kappa$ is positive definite for $\kappa\geq 1$, it is convenient to employ the factorization
\begin{equation}\label{E:middle}
(\Lc_q+\ka)=(\Lc_0+\ka)^{\frac12}\bigl[1\mp S(q)^* S(q) \bigr] (\Lc_0+\ka)^{\frac12}
\qtq{with} S(q) := C_+ \overline q (\Lc_0+\ka)^{-\frac12}
\end{equation}
and $\kappa\geq 1$.  We will also employ the factorizations
\begin{equation}\label{E:rightleft}
(\Lc_q+\ka)=\bigl[1\mp T(q,q)\bigr] (\Lc_0+\ka) 
\qtq{and}
(\Lc_q+\ka)=(\Lc_0+\ka) \bigl[1\mp T(q,q)^*\bigr], 
\end{equation} 
where we have introduced two more operators on $L^2_+$,
\begin{equation}\label{E:T&T*}
T(q,g)=qC_+\bar{g}(\Lc_0+\ka)^{-1} \qtq{and} T(q,g)^*=(\Lc_0+\ka)^{-1} gC_+\bar{q}.
\end{equation} 
The first factorization in \eqref{E:rightleft} will be used to prove \eqref{E:equiv} when $s=1$, while the second will be used when $s=-1$.  The remaining cases will follow by interpolation.

Our first application of the Carleman inequality is to bound the norm of the operator $S(q)$ defined in \eqref{E:middle}.  We also include here a second inequality \eqref{E:op bound} that will be crucial for proving nondegeneracy of the symplectic form \eqref{eqn:omega}.

\begin{corollary}
For $q\in L^2_+$ and $\kappa\geq 1$,
\begin{equation}\label{E:S bound}
\| S(q)\|_{L^2_+\to L^2_+}^2 = \| S(q)^* \|_{L^2_+\to L^2_+}^2 \leq \tfrac{1}{2\pi} \|q\|_{L^2_+}^2.
%
\end{equation}
Moreover, recalling the operator $\tilde\partial^{-1}$ from \eqref{delta inverse} and \eqref{delta inverse'}, we have
\begin{equation}\label{E:op bound}
\langle g,C_+ \overline q (-i\tilde{\dd})^{-1} qg\rangle \leq \tfrac{1}{2\pi} \|q\|_{L^2_+}^2 \|g\|_{L^2_+}^2 \quad\text{for all $g\in (L^2_-)^\perp$}.
\end{equation}
\end{corollary}

\begin{proof}
By the method of $TT^*$, we see that \eqref{E:S bound} is equivalent to
\begin{equation}\label{E:S bound'}
\langle g,  C_+ \overline q (\Lc_0+\ka)^{-1} qg\rangle \leq \tfrac{1}{2\pi} \|q\|_{L^2_+}^2 \|g\|_{L^2_+}^2,
\end{equation}
which reveals its similarity to \eqref{E:op bound}.

To understand the significance of Carleman's inequalities in this context, it is convenient to write out their left-hand sides in terms of $\widehat f$.  In the line case, this calculation was performed in \eqref{CarlFub}.  Performing the parallel analysis on the circle, one finds that Lemma~\ref{L:Carleman} is equivalent to the statements 
\begin{align}\label{E:Carleman'}
	\langle f, (2\Lc_0+2)^{-1} f \rangle \leq \tfrac{1}{4\pi} \bigl\| f\bigr\|_{L^1_+(\T)}^2 \qtq{and}
	\langle f, (2\Lc_0)^{-1} f \rangle \leq \tfrac{1}{4\pi} \bigl\| f\bigr\|_{L^1_+(\R)}^2.
\end{align}
Choosing $f=gq$ and applying H\"older's inequality then settles all claims except the torus case of \eqref{E:op bound}.  For this last case, we choose $f(x) = e^{-ix}q(x)g(x)$, which belongs to $L^1_+(\T)$ because $g\in (L^2_-)^\perp$.
\end{proof}

Continuing with our goal of proving Theorem~\ref{T:equiv} using the factorizations \eqref{E:middle} and \eqref{E:rightleft}, we need to show invertibility of each of the operators in square brackets.  To do this, we will use Fredholm theory and so must show that the operators are compact.

\begin{lemma}\label{L:SST}
For $q\in L^2_+$ and $\kappa\geq 1$, the operators $T(q,q)$, $T(q,q)^*$, $S(q)^*S(q)$, and $S(q)S(q)^*$ are all Hilbert--Schmidt and have identical spectra.
\end{lemma}

\begin{proof}
We begin by proving quantitative bounds on the Hilbert--Schmidt norm:
\begin{equation}\begin{aligned}\label{E:QHS}
&\|T(q,g)\|_{\I _2}= \|T(q,g)^*\|_{\I _2} \lesssim \|q\|_{L^2_+}\|g\|_{L^2_+} \\ 
&\|S(q)^*S(q)\|_{\I _2} = \|S(q)S(q)^*\|_{\I _2} \lesssim \|q\|_{L^2_+}^2 .
\end{aligned}\end{equation}
The second estimate was demonstrated previously in \cite{FrankRead}; however, their method does not seem to apply to $T$ or $T^*$.

We will exploit the $L^2$-boundedness of certain `diagonal' maximal operators in $\Z^2$ and $\R^2$, specifically,
\begin{align*}
[\mathcal{M} f](n,m)=\sup_{h\in\N}\ \tfrac{1}{2h+1}\sum_{\ell=-h}^h|f(n-\ell,m-\ell)| &\qtq{obeys} \|\mathcal{M}f\|_{\ell^2(\Z^2)}\lesssim \|f\|_{\ell^2(\Z^2)},\\
[\mathcal{M} f](x,y)=\sup_{h>0}\ \tfrac{1}{2h}\int_{-h}^h|f(x-\ell,y-\ell)|\,d\ell &\qtq{obeys} \|\mathcal{M}f\|_{L^2(\R^2)}\lesssim \|f\|_{L^2(\R^2)}.
\end{align*}
These bounds can be proved, for example, by combining boundedness of the one-dimensional maximal operator with a coordinate rotation.  In the discrete case, it is helpful to regard $\Z^2$ as an infinite chessboard and notice that $\mathcal{M}$ operates on the black and white sublattices independently.

To avoid repetition, we will present the key argument only on the torus.  Employing a Fourier basis, we have
\begin{align*}
\bigl\|T(q,g)\bigr\|_{\I _2}^2&=  \sum_{n,m \geq 0}\,\biggl|\mkern 2mu \sum_{\ell=0}^{n\wedge m} \widehat{q\mkern0.5mu} (n-\ell) 
	\mkern 2mu\overline{\widehat g(m-\ell)} \mkern 2mu (m+\ka)^{-1} \biggr|^2 \\
&\leq \sum_{n,m \geq 0}\,\biggl| \f{1}{1+n\wedge m}\sum_{\ell=0}^{n\wedge m} \widehat{q}(n-\ell)\overline{\widehat{g}(m-\ell)}\biggr|^2 \\
&\leq \sum_{n,m\geq 0} \, \Bigl| [\mathcal{M} \,\widehat{q\mkern0.5mu} \otimes \overline{\widehat g}\mkern2mu] (n,m)\Bigr|^2\\
&\lesssim  \| \widehat{q\mkern0.5mu} \otimes \overline{\widehat g}\mkern1mu \|_{\ell^2(\Z^2)}^2 \lesssim  \|q\|_{L^2_+}^2\|g\|_{L^2_+}^2 .
\end{align*}

A parallel argument can be used to bound
\begin{align}\label{1:52}
\|S(q)^*S(q)\|_{\I _2}\lesssim \|q\|_{L^2_+}^2 \qtq{or equivalently} \|S(q)\|_{\I _4}=\|S(q)^*\|_{\I _4}\lesssim \|q\|_{L^2_+}.
\end{align}
Here $\I_4$ denotes the corresponding trace ideal; see \cite{simon_book}.

The fact that $T$ and $T^*$ have the same Hilbert--Schmidt norm can be seen, for example, by cycling the trace: $\Tr(TT^*)=\Tr(T^*T)$.  By the same reasoning, $S^*S$ and $SS^*$ have the same Hilbert--Schmidt norm.

It remains to show that all operators have the same spectrum.  Our key tool is the relation
\begin{equation}\label{ABBA}
\sigma(AB)\cup\{0\}  = \sigma(BA)\cup\{0\} \quad\text{for any pair of bounded operators $A,B$.}
\end{equation}
For a proof of this, as well as a historical discussion, see \cite{Deift}.  In our applications, $AB$ and $BA$ will both be compact; correspondingly, zero is automatically in the spectrum and $\sigma(AB) = \sigma(BA)$.

The fact that $\sigma( S(q)^*S(q) )=\sigma( S(q)S(q)^* )$ is an easy application of \eqref{ABBA}.  In order to see that $\sigma( T(q,q) ) = \sigma( S(q)S(q)^* )$, we would like to apply \eqref{ABBA} with $A:f\mapsto C_+ \widebar q (\Lc_0 + \kappa)^{-1}f$ and $B:f\mapsto qf$; however, $B$ is an unbounded operator for general $q\in L^2_+$.  To address this, we first consider $q_{\leq N}$, which belongs to $L^\infty_+$, and then send $N\to\infty$.  Notice that \eqref{E:QHS} guarantees convergence in Hilbert--Schmidt norm, which guarantees convergence of the (Hilbert-regularized) Fredholm determinant and so of the spectrum.

Using the fact that $S(q)S(q)^*$ is selfadjoint, we have
$$
\sigma( T(q,q)^* )=\overline{\sigma( T(q,q) )} = \overline{\sigma( S(q)S(q)^* )} = \sigma( S(q)S(q)^* ).
$$
Here, the overline indicates (pointwise) complex conjugation on the set.
\end{proof}

The operators $T(q,q)$ and $T(q,q)^*$ are not selfadjoint and so one cannot expect to control the norm of their resolvents solely in terms of the location of their spectra.  
Fortunately, our operators are Hilbert--Schmidt, which limits the intensity of this pseudospectral phenomenon.  Concretely, we will rely on the elegant quantitative treatment by Bandtlow \cite{Bandtlow}.  Combining Theorem 4.1 and Proposition 3.8 from \cite{Bandtlow} yields the following:

\begin{proposition}[\cite{Bandtlow}]\label{P:Bandtlow}
There is a constant $a>0$ so that  
$$
\bigl\|(I-A)^{-1}\bigr\|_\op \lesssim \f{1}{\dist(1,\sg(A))}\exp\Bigl(\f{a}{\dist(1,\sigma(A))^2} \bigl\|A\bigr\|_{\I_2}^2 \Bigr)
$$
for any Hilbert--Schmidt operator $A$.
\end{proposition}

\begin{proof}[Proof of Theorem~\ref{T:equiv}]
From \eqref{E:S bound} we see that
\begin{equation}\label{<S<}
1-\tfrac{1}{2\pi}M(q) \leq 1 - S(q)^* S(q) \leq 1 \leq 1 + S(q)^* S(q) \leq 1+\tfrac{1}{2\pi}M(q)
\end{equation}
in the sense of quadratic forms.  Combining this with \eqref{E:middle} we see that $\Lc_q+\kappa$ is both positive definite and boundedly invertible when $q\in B_{M_*}$ and $\kappa\geq 1$.

The combination of \eqref{<S<} and \eqref{E:middle} suffices to prove \eqref{E:equiv} when $s=\pm\frac12$.  To treat $s=\pm1$, we must consider $T(q,q)$ and $T(q,q)^*$.  Our key claim about these operators is the following: If $M<M_*$, then there is a constant $C_M$ so that 
\begin{equation}\label{373}\begin{gathered}
\bigl\| 1 \mp T^*\bigr\|_\op  =  \bigl\| 1 \mp T \bigr\|_\op \leq C_M \\
\bigl\| [1 \mp T^*]^{-1} \bigr\|_\op  =  \bigl\| [1 \mp T]^{-1} \bigr\|_\op \leq C_M.
\end{gathered}\end{equation}
The first claim follows immediately from \eqref{E:QHS}.  Controlling the inverses is more subtle and relies on Proposition~\ref{P:Bandtlow} with the required control on the spectrum provided by Lemma~\ref{L:SST} and \eqref{<S<}.

By rewriting \eqref{E:rightleft} as
\begin{alignat*}{2}
(\Lc_q+\ka)&=\bigl[1\mp T(q,q)\bigr] (\Lc_0+\ka),
	\quad& (\Lc_q+\ka)^{-1}&=\bigl[1\mp T(q,q)^*\bigr]^{-1} (\Lc_0+\ka)^{-1}, \\
(\Lc_0+\ka) &= \bigl[1\mp T(q,q)\bigr]^{-1} (\Lc_q+\ka),
	\quad& (\Lc_0+\ka)^{-1} &= \bigl[1\mp T(q,q)^*\bigr] (\Lc_q+\ka)^{-1}, 
\end{alignat*}
we see that \eqref{373} implies \eqref{E:equiv} when $s=\pm1$. The remaining cases, $-1<s<1$,  follow by complex interpolation; see \cite[Ch.~1]{LionsMagenes}.
\end{proof}

The basic conservation laws for the continuum Calogero--Moser models are expressed through quadratic forms based on the Lax operator.  However, the Lax operator changes as $q(t)$ evolves and so Theorem~\ref{T:equiv} plays the important role of connecting these quantities to the time-independent (and very familiar) $H^s$ norms.

In this paper, we wish to study the continuum Calogero--Moser models in the scale-invariant space $L^2_+$.  This corresponds to the $s=0$ case of \eqref{E:equiv}, which is absolutely trivial.  Nevertheless, Theorem~\ref{T:equiv} has an important role to play even in this setting:  We will use it to preclude frequency cascade behaviour or, what is equivalent, to show that orbits are equicontinuous.  To do this, we develop an expression of traditional $L^2$-equicontinuity based on the functional calculus of the operators $\Lc_q$; see Theorem~\ref{T:equi-c} below.  Let us begin by recalling the traditional notion:

\begin{definition}\label{D:equi}
A bounded set $Q\subseteq L^2$ is \emph{$L^2$-equicontinuous} if and only if
\begin{equation}\label{equi:defn1}
\lim_{\delta\to 0}\ \sup_{q\in Q}\ \sup_{|y|<\delta} \ \int \bigl| q(x+y) -q(x)\bigr|^2 \,dx = 0.
\end{equation}
\end{definition}

It is well known that \eqref{equi:defn1} is equivalent to $L^2$-tightness of the corresponding set of Fourier transforms, $\{\widehat q: q \in Q\}$, which in turn is equivalent to tightness of the family of measures $\{ |\widehat q(\xi)|^2\,d\xi : q\in Q\}$.  Our perspective is this: Think of these measures as the spectral measures associated to the operator $-i\partial_x$ and the vectors $q\in Q$.  (We consider $Q\subset L^2_+$ and so may use the operator $\Lc_0$ instead.)  In this way, we see that the proper adaptation of equicontinuity to some new operator $A$ is to demand that spectral measures are tight.

We thought it was important to share this intuition that underlies what follows because expository efficiency demands a rather different approach.  First, we will employ the functional calculus, rather than spectral measures.  For example, we express tightness of $\Lc_0$ spectral measures as \eqref{equi2}.  Second, we will be dealing with both a set of vectors $Q$ and a family of operators $\{\Lc_q : q\in Q\}$.  Taking an abstract approach would then compel us to introduce the notion of \emph{uniformly} spectrally tight across a \emph{collection} of operators.  What we really need is the following:  
%
%

\begin{theorem}\label{T:equi-c}
Fix $M<M_*$.   For $Q \subset B_M$, the following are equivalent:
\begin{gather}
\label{equi1} \text{$Q$ is $L^2$-equicontinuous}, \\
\label{equi2} \lim_{\kappa\to\infty}\ \sup_{q\in Q}\ \bigl\langle q, \tfrac{(\Lc_0+1)^2}{(\Lc_0+1)^2+\kappa^2} q \bigr\rangle =0, \\
\label{equi3} \lim_{\kappa\to\infty}\ \sup_{q\in Q}\ \bigl\langle q, \tfrac{(\Lc_q+1)^2}{(\Lc_q+1)^2+\kappa^2} q \bigr\rangle =0.
\end{gather}
\end{theorem}

\begin{proof}
The equivalence of \eqref{equi1} and \eqref{equi2} is quite elementary; details can be found, for example, in \cite[\S4]{killip2019kdv}.
Note that the meaning of \eqref{equi2} would be unchanged if we replaced all instances of $(\Lc_0+1)$ by $\Lc_0$.  Likewise $(\Lc_q+1)$ could be replaced by $\Lc_q$ in \eqref{equi3}.  We chose the form given above so as to exploit the following consequence of Theorem~\ref{T:equiv}: There is a constant $C_M$ so that
\begin{equation}\label{E:square_equiv}
C_M^{-2} (\Lc_0+1)^2 \leq (\Lc_q+1)^2 \leq C_M^2 (\Lc_0+1)^2 \quad\text{for all $q\in Q$}
\end{equation}
in the sense of quadratic forms.

For each $\kappa>0$, the function
\begin{equation*}
X \mapsto \tfrac{X}{X+\kappa^2} = 1 - \tfrac{\kappa^2}{X+\kappa^2}
\end{equation*}
is operator increasing on the class of positive-definite operators, because inversion is operator decreasing.  (For the general theory of operator monotone functions, see \cite{Simon:Loewner}).  Applying this monotonicity to \eqref{E:square_equiv}, we deduce that 
\begin{equation}\label{E:square_equiv'}
\tfrac{(\Lc_0+1)^2}{(\Lc_0+1)^2 + C_M^{2} \kappa^2} \leq \tfrac{(\Lc_q+1)^2}{(\Lc_q+1)^2 + \kappa^2} \leq \tfrac{(\Lc_0+1)^2}{(\Lc_0+1)^2 + C_M^{-2} \kappa^2}.
\end{equation}
The equivalence of \eqref{equi2} and \eqref{equi3} follows immediately.
\end{proof}

Later, in Section~\ref{sec:GWP_kappa}, we will use Theorem~\ref{T:equi-c} to propagate equicontinuity from a set of initial data to the ensemble of trajectories emanating from it.  One central application of this equicontinuity is to control the difference flow.  As we are working at the scaling-critical level, one must accept that the rate of convergence is governed solely by the equicontinuity properties of the initial data.  The final result of this section contains some of the technical manifestations of equicontinuity we will employ to control the difference flow.

\begin{corollary}\label{C:decay}
Fix $M<M_*$ and suppose $Q \subset B_M$ is equicontinuous.  Then
\begin{align}
\lim_{\ka\to\infty} \sup_{q\in Q}\ \|\Lc_q m(\ka,q)\|_{L^2_+}=0,\label{item2}\\
\lim_{\ka\to\infty}\sup_{q\in Q}\|\ka \Lc_q (\Lc_q+\ka)^{-1}m(\ka,q)\|_{L^2_+}=0,\label{item4}
\end{align}
where $m(\kappa,q) = (\Lc_q+\ka)^{-1}q$.
\end{corollary}

\begin{proof}
For each $q\in Q$ and $\kappa\geq 2$,
\begin{align*}
\|\Lc_q m(\ka,q)\|_{L^2_+}^2 = \bigl\langle q,\tfrac{\Lc_q^2}{(\Lc_q+\kappa)^2} q \bigr\rangle
	\leq 4 \bigl\langle q,\tfrac{(\Lc_q+1)^2+1}{(\Lc_q+1)^2+\kappa^2} q \bigr\rangle
	\leq 4 \bigl\langle q,\tfrac{(\Lc_q+1)^2}{(\Lc_q+1)^2+\kappa^2} q \bigr\rangle + \tfrac{4}{\kappa^2} M ,
\end{align*}
simply because $\Lc_q+1$ is positive definite and $Q \subset B_M$.  The claim \eqref{item2} now follows from Theorem~\ref{T:equi-c} and the assumed equicontinuity of $Q$.


As $\Lc_q+1$ is positive definite, $\kappa\geq 2$ implies 
\begin{align*}
\|\ka \Lc_q (\Lc_q+\ka)^{-1}m(\ka,q)\|_{L^2_+} \leq 2  \| \Lc_q m(\kappa, q)\|_{L^2_+}.
\end{align*}
Thus \eqref{item4} also follows from \eqref{item2}.
\end{proof}

\section{Construction of the symplectic form}\label{sec:construction}

The main goal of this section is to show that \eqref{eqn:omega} defines a strong symplectic form on $B_{M_*}$ in both geometries.  Recall
\begin{equation*}
\om_q(f,g)=\Re\lan f,\Om (q)g\ran, \ \ \Om (q)g:=i(1\mp2 C_+ \Theta(q))g,\ \ \text{and} \ \   \Theta(q)g:=iq\tilde{\dd}^{-1}\Re(\bar{q}g).
\end{equation*}
The fact that $\omega$ is \emph{strong} guarantees that there is a bounded linear mapping from the Wirtinger derivative of the Hamiltonian to the corresponding Hamiltonian vector field; see \eqref{J} and \eqref{symplectic gradient}.

In the line setting, we already know that \eqref{eqn:omega_hash} defines a strong symplectic form on $L^2(\R)$.  This can be used to infer that $\omega$ is a closed $2$-form on $L^2_+(\R)$; however, it does not address the key question: nondegeneracy.  Indeed, \eqref{eqn:omega_hash} is a strong symplectic form throughout $L^2(\R)$, while its restriction to $L^2_+(\R)$ does become degenerate at large mass in the focusing case. 

As we also wish to treat the torus case, we will be developing our analysis from scratch, without relying on the observations of the previous paragraph.

\begin{theorem}\label{T:symplectic form}
The form $\om$ defined in \eqref{eqn:omega} is a strong symplectic form on $B_{M_*}$. In the focusing case, the form is degenerate on any ball $B_M$ with $M>M_*$.
\end{theorem}

\begin{proof}
As discussed in Section~\ref{S:Note}, $\tilde\partial^{-1}$ maps $L^1\to L^\infty$; thus, $\Theta(q)$ is an $L^2$-bounded $\R$-linear operator and so $\omega_q$ is well-defined throughout $L^2_+\times L^2_+$.  Let us also observe, for future use, that $\Theta(q)^2=0$ and so
\begin{equation}\label{eqn:S_inverse}
(1\mp2 \Theta(q))^{-1}=1\pm2 \Theta(q).
\end{equation}

To prove that $\om$ is a symplectic form, we have to verify that it is bilinear, antisymmetric, nondegenerate, and closed. We will demonstrate each of these properties successively.  The mass restriction will only enter into the discussion of nondegeneracy.

The fact that $\om$ is a bilinear form follows immediately from the definition \eqref{eqn:omega}, which we may rewrite as
\begin{align*}
\om_q(f,g)=\Re\,\lan f,\Om (q)g\ran = \Re\, \lan f,ig\ran \pm 2\Re\, \bigl\lan f,C_+\bigl[q\tilde{\dd}^{-1}\Re(\bar{q}g) \bigr]\bigr\ran.
\end{align*}

One readily sees that 
\begin{align*}
 \Re\, \lan f,ig\ran  = - \Re\, \lan g, if\ran.
\end{align*}
Moreover, using the fact that $\tilde{\dd}^{-1}$ is an anti-selfadjoint operator which maps real-valued functions to real-valued functions, we have
\begin{align*}
\Re\,\bigl\lan f, C_+\bigl[q\tilde{\dd}^{-1}\Re(\bar{q}g)\bigr]\bigr\ran&=\bigl\lan\Re (\bar{q}f),\tilde{\dd}^{-1}\Re(\bar{q}g)\bigr\ran\\
&=-\bigl\lan \tilde{\dd}^{-1}\Re(\bar{q}f),\Re(\bar{q}g)\bigr\ran\\
&=-\Re\,\bigl\lan q\tilde{\dd}^{-1}\Re(\bar{q}f),g\bigr\ran\\
&=-\Re\, \bigl\lan g, C_+\bigl[q\tilde{\dd}^{-1}\Re(\bar{q}f)\bigr]\bigr\ran.
\end{align*}
Combining the last two displays, we deduce that  
$$
\om_q(f,g)=-\om_q(g,f),
$$
that is, $\om$ is an antisymmetric 2-form.

We now turn to the question of nondegeneracy of $\om$.  Recall that $\om$ is said to be nondegenerate at a point $q\in L^2_+$ if the following condition holds:
$$
\om_q(f,g)= 0 \text{ for all $f\in L^2_+$} \qquad\text{if and only if} \qquad g=0.
$$
As a first step in this direction, we observe the following equivalent condition for degeneracy:

\begin{lemma}\label{lem:lem3.2}
The form $\om$ is degenerate at a point $q\in L^2_+$ if and only if there exists $G\in (L^2_-)^{\perp}\setminus\{0\}$ such that
\begin{align}\label{2:50}
G\mp i\bar{q}\tilde{\dd}^{-1}(qG)\in L^2_-.
\end{align}
\end{lemma}

\begin{proof}
Recalling \eqref{eqn:omega}, we see that $\om$ is degenerate at a point $q\in L^2_+$ if and only if there exists $g\in L^2_+\backslash\{0\}$ such that
\begin{align}\label{2:51}
\Om (q)g=0 \qtq{or equivalently,} (1\mp 2 \Theta (q))g\in (L^2_+)^{\perp}.
\end{align}

If such a $g$ exists, we define $G:=\overline{(1\mp2\Theta(q))g}\in (L^2_-)^{\perp}$. Then using \eqref{eqn:S_inverse} and the property
\begin{equation}\label{eqn:bar_S} 
\overline{\Theta(g)f}= - \Theta(\bar{g})\bar{f},
\end{equation}
we obtain
\begin{equation}\label{eqn:gbar}
\bar{g}=(1\mp 2\Theta(\bar{q})){G}={G}\mp i\bar{q}\tilde{\dd}^{-1}(q{G}+\bar{q}\widebar{G})\in L^2_-.
\end{equation}
In particular, as $g\neq 0$ we must have that $G\neq0$. Moreover, as $q\in L^2_+$ and $G\in (L^2_-)^{\perp}\subseteq L^2_+$, we have $\bar{q}\tilde{\dd}^{-1}(\bar{q}\widebar{G})\in L^2_-$.  Thus, \eqref{2:50} holds.

Conversely, if there exists $G\in (L^2_-)^{\perp}\setminus\{0\}$ satisfying \eqref{2:50}, then we define
$$
g:=\overline{(1\mp 2\Theta(\bar q))G}.
$$
Using \eqref{2:50} and the fact that $q\in L^2_+$ and $G\in (L^2_-)^{\perp}\subseteq L^2_+$, we deduce that 
$$
\bar g = G\mp i\bar q \tilde{\dd}^{-1}(gG) \mp i\bar q \tilde{\dd}^{-1}(\bar{q}\widebar{G}) \in L^2_-
$$
and so $g\in L^2_+\setminus\{0\}$.  Finally, combining \eqref{eqn:S_inverse} and \eqref{eqn:bar_S}, we find that
$$
(1\mp 2\Theta(q))g=\widebar{G}\in (L^2_+)^\perp,
$$
which, in view of \eqref{2:51}, shows that $\om$ is degenerate at $q$.
\end{proof}

Equipped with Lemma~\ref{lem:lem3.2}, we are now ready to address the (non)-degeneracy of $\om$. We first prove that $\om$ is nondegenerate on $B_{M_*}$.  To this end, we fix $q\in B_{M_*}$ and suppose that $G\in (L^2_-)^{\perp}$ satisfies \eqref{2:50}. We will show that necessarily $G=0$, which by Lemma~\ref{lem:lem3.2} yields the nondegeneracy of $\om$.

We first observe that \eqref{2:50} guarantees 
\begin{equation}\label{eqn:degen}
0=\lan G, G\mp i\bar{q}\tilde{\dd}^{-1}(qG)\ran=\|G\|_{L^2_+}^2\mp\|\tilde{\dd}^{-1}(qG)\|_{\dot{H}^{1/2}}^2.
\end{equation}
In the defocusing case, the right-hand side of \eqref{eqn:degen} is clearly positive definite, and so we must have $G=0$. In the focusing case, we apply \eqref{E:op bound} to obtain
$$
0=\lan G,G\ran -\lan G,i\bar{q}\tilde{\dd}^{-1}(qG)\ran\geq\bigl(1-\tfrac1{2\pi}\|q\|_{L^2_+}^2\bigr)\|G\|_{(L^2_-)^{\perp}}^2.
$$
If $q\in B_{M_*}$, this is only possible for $G=0$.  This completes the proof that $\om$ is nondegenerate on $B_{M_\ast}$.

Next, we show by explicit counterexample that in the focusing case, $\om$ is degenerate on every ball $B_M$ with $M>M_\ast$. In the case of the torus, we consider the point $q\equiv1$ and take $G(x)=e^{ix}$.  For the problem posed on the real line, we take $q(x)=G(x) = \sqrt{2} (x+i)^{-1}$.  In both cases, $\|q\|^2_{L^2_+}=2\pi$, $G\in(L^2_-)^{\perp}\backslash\{0\}$, and
$$
G-i\bar{q}\tilde{\dd}^{-1}(qG)\in L^2_-.
$$
It thus follows from Lemma \ref{lem:lem3.2} that the form $\omega$ is degenerate at $q$.  

To conclude that $\om$ is a symplectic form, it remains to show that it is closed.  To this end, it suffices to show (cf. \cite{lang1995differential}) that for all tangent vectors $f,g,h\in L^2_+$ we have
\begin{align}\label{2:52}
D_f(\om(g,h))+D_g(\om(h,f))+D_h(\om(f,g))=0.
\end{align}

For fixed $q,f,g,h\in L^2_+$, using the fact that $\tilde{\dd}^{-1}$ is an anti-selfadjoint operator which maps real-valued functions to real-valued functions, we obtain
\begin{align*}
\bigl[D_f(\om(g,h))\bigr](q)
&=\mp\lim_{\eps\to0}\f{\Re\,\lan g,2i[\Theta(q+\eps f)-\Theta(q)]h\ran}{\eps}\\
&=\pm2\Re\,\lan g,f\tilde{\dd}^{-1}\Re(\bar{q}h)+q\tilde{\dd}^{-1}\Re(\bar{f}h)\ran\\
&=\pm2\,\lan\Re(g\bar{f}),\tilde{\dd}^{-1}\Re(\bar{q}h)\ran \mp 2\,\lan\Re(f\bar{h}),\tilde{\dd}^{-1}\Re(\bar{q}g)\ran.
\end{align*}
Thus,
\begin{align*}
\bigl[D_f(\om(g,h))\bigr](q)&+\bigl[D_g(\om(h,f))\bigr](q)+\bigl[D_h(\om(f,g))\bigr](q)\nonumber\\
&=\pm2\,\lan\Re(g\bar{f}),\tilde{\dd}^{-1}\Re(\bar{q}h)\ran\mp2\,\lan\Re(f\bar{h}),\tilde{\dd}^{-1}\Re(\bar{q}g)\ran\\
&\quad\pm2\,\lan\Re(h\bar{g}),\tilde{\dd}^{-1}\Re(\bar{q}f)\ran\mp2\,\lan\Re(g\bar{f}),\tilde{\dd}^{-1}\Re(\bar{q}h)\ran \\
&\quad\pm2\,\lan\Re(f\bar{h}),\tilde{\dd}^{-1}\Re(\bar{q}g)\ran\mp2\,\lan\Re(h\bar{g}),\tilde{\dd}^{-1}\Re(\bar{q}f)\ran =0,
\end{align*}
which settles \eqref{2:52}.

To complete the proof of Theorem~\ref{T:symplectic form}, it remains to show that $\om$ is a \emph{strong} form on $B_{M_*}$.  This amounts to showing (cf. \cite{chernoff2006properties}) the following:

\begin{proposition}\label{P:J}
For each $q\in B_{M_*}$, there is a bounded and boundedly invertible $\R$-linear map $J(q):L^2_+\to L^2_+$ satisfying
\begin{align}\label{J}
\om_q(f, J(q)g) = \Re \,\langle f, g\rangle \qtq{for all} f,g\in L^2_+.
\end{align}
\end{proposition}

\begin{proof}
On the basis of our definition \eqref{eqn:omega} of $\om_q$, such a $J(q)$ would satisfy
$$
\Re\, \langle f, g\rangle = \Re\, \langle f, \Om(q) J(q) g\rangle \qtq{for all} f,g\in L^2_+
$$ 
and consequently, 
\begin{align}\label{J inv}
\Om(q) J(q) = \Id \qtq{for all} q\in B_{M_*}.
\end{align}
Thus, to prove the proposition, it suffices to show that $\Om(q)$ is boundedly invertible.  In turn, this will follow from the Fredholm alternative and the observation that $-i\Om(q)= 1\mp 2 C_+ \Theta(q)$ is a compact perturbation of the identity with $\ker\Om(q)=\{0\}$.

That $\ker\Om(q)=\{0\}$ follows from the nondegeneracy of $\om$ on $B_{M_*}$.  On the other hand, using \eqref{delta inverse}, we may write
\begin{align*}
\Theta(q)g&=\tfrac{i}{4}\!\int_{0}^{2\pi}\!\bigl(1-\tfrac{y}{\pi}\bigr)q(x)(\bar{q}g)(x-y)\,dy +\tfrac{i}{4}\!\int_{0}^{2\pi}\!\bigl(1-\tfrac{y}{\pi}\bigr)q(x)(q\bar{g})(x-y)\,dy \quad\text{on $\T$},\\
\Theta(q)g&=\tfrac{i}{4}\!\int_\R\sgn(x-y)q(x)\bar{q}(y) g(y) \,dy+\tfrac{i}{4}\!\int_\R\sgn(x-y) q(x)q(y)\bar{g}(y)\,dy \qquad \, \text{on $\R$}.
\end{align*}
The integral kernels above are in $L^2_{x,y}$ and so correspond to Hilbert--Schmidt operators. As the map $g\mapsto \bar{g}$ is bounded on $L^2_+$ and the Cauchy--Szeg\H{o} projection $ C_+:L^2\rightarrow L^2_+$ is also bounded, it follows that $C_+ \Theta(q)$ is a compact operator on $L^2_+$ and so $-i\Om(q)= 1\mp 2 C_+ \Theta(q)$ is a compact perturbation of the identity.
\end{proof}

Proposition~\ref{P:J} shows that $\omega_q$ is a \emph{strong} symplectic form and the proof of Theorem~\ref{T:symplectic form} is complete.
\end{proof}

From the perspective of Hamiltonian mechanics, the map $J(q)$ in \eqref{J} provides the link between the derivative of the Hamiltonian and the corresponding Hamiltonian vector field.  In view of \eqref{1:51}, this takes the form
\begin{align}\label{symplectic gradient}
\na_\om F(q)=2J(q) \tfrac{\dl F}{\dl \bar{q}}\bigr|_q
\end{align}
for any real-valued observable $F$.  Concomitant with this, $J(q)$ has the important role of defining the Poisson bracket:
\begin{equation}\label{PBJ}
\{F, G\} (q) := \omega_q( \nabla_\omega G,  \nabla_\omega F ) = 4 \Re \big\langle \tfrac{\dl F}{\dl \bar{q}}, J(q) \tfrac{\dl G}{\dl \bar{q}}  \bigr\rangle.
\end{equation}
Note that the reversal of $F$ and $G$ in the first equality is deliberate; we are following the sign conventions of \cite{MR2269239}.

The main difficulty in working with the symplectic form $\om$ on $L^2_+$, as opposed to \eqref{eqn:omega_hash} on $L^2$, is that the map $J(q)$ does not have an explicit formula. Fortunately, for quantities of interest to the continuum Calogero--Moser models, explicit formulae for the symplectic gradient are achievable, as we shall see in Section~\ref{sec:flows}.

\section{The Hamiltonian and commuting flows}\label{sec:flows}

The goal of this section is to complete the proofs of Theorems~\ref{T:R} and~\ref{T:T} by verifying that the Hamiltonians \eqref{eqn:H(q) R} and \eqref{H on T} generate the flows \eqref{CCMR} and \eqref{CCMT}, respectively, under the symplectic form defined in \eqref{eqn:omega}.  We will also derive the Hamiltonian flows corresponding to the generating function $\beta(\kappa, q)$ and the conserved quantities $\mathcal{E}_n(q)$ and present Lax pair formulations attendant to all these flows.

Our point of departure is the generating function
$$
\beta(\kappa,q) = \langle q, m(\ka,q)\rangle \qtq{where} m(\ka,q) = (\Lc_q  +\kappa)^{-1} q.
$$
By Theorem~\ref{T:equiv}, both $m(\ka,q)$ and $\beta(\ka,q)$ are well defined for all $\kappa\geq 1$ and $q\in B_{M_*}$. To ease notation, in what follows we will often use the shorthand $\beta_\ka$ for $\beta(\kappa,q)$ and $m_\ka$ for $m(\ka,q)$.

\begin{proposition}\label{lem:beta_results}
Fix $\ka\geq 1$.  Throughout $B_{M_\ast}$ we have
\begin{equation}\label{eqn:beta_deriv}
\f{\dl \beta_\ka}{\dl \bar{q}}=m_\ka\pm C_+\bigl[m_\ka C_-(q\overline m_\ka )\bigr]
\end{equation}
and the symplectic gradient of $\beta_\ka$ is given by
\begin{equation}\label{beta_symplectic_deriv}
\begin{aligned}
\na_\om\beta_\ka&=-2im_\ka\pm 2im_\ka C_+(q\overline m_\ka) \qquad\qquad\qquad\qquad\qquad \quad\,\text{on $\R$,}\\
\na_\om\beta_\ka&=-2im_\ka\pm 2im_\ka C_+(q\overline m_\ka)\mp\tfrac{i}{\pi} m_\ka \beta_\ka\mp\tfrac{i}{\pi}q\|m_\ka \|^2_{L^2} \quad\,\text{on $\T$}.
\end{aligned}
\end{equation}
The associated Hamiltonian flow 
\begin{align}\label{beta flow}
\dd_t q=\na_\om \beta_\ka(q)
\end{align}
admits the following Lax pair representation
\begin{equation}
\text{$q$ solves \eqref{beta flow}} \qquad \iff \qquad \dd_t \Lc_q = \bigl[\mathcal{P}^\ka_\beta,\Lc_q\bigr]\label{beta Lax pair}
\end{equation}
where
\begin{equation}\label{eqn:Peter_op_ka}
\begin{aligned}
\mathcal{P}^\ka_\beta&:=-2i(\Lc_q+\ka)^{-1}\pm 2im_\ka C_+ \overline m_\ka \qquad\qquad\qquad\qquad\qquad\qquad\quad \,\, \text{on $\R$,}\\
\mathcal{P}^\ka_\beta&:=-2i(\Lc_q+\ka)^{-1}\pm 2im_\ka C_+ \overline m_\ka\mp\tfrac{i}{\pi}\beta_\ka(\Lc_q+\ka)^{-1} \mp\tfrac{i}{\pi}\|m_\ka\|_2^2  \quad\text{on $\T$}.
\end{aligned}
\end{equation}
Finally, the Peter operator $\mathcal{P}^\ka_\beta$ enjoys the following special property:
\begin{equation}
\text{$q$ solves \eqref{beta flow}} \qquad \iff \qquad \dd_t q=\mathcal{P}^\ka_\beta q.\label{one true peter}
\end{equation}
\end{proposition}

\begin{proof}
The formula \eqref{eqn:beta_deriv} is obtained via a straightforward calculation using the resolvent identity and the definition \eqref{Wirtinger} of the Wirtinger derivative.

We now turn to \eqref{beta_symplectic_deriv}. We present the details in the case of the circle, which is more computationally involved than the line case.  Indeed, the sole distinction between the two geometries arises from \eqref{Cid}.

In view of \eqref{J inv}, \eqref{symplectic gradient}, and \eqref{eqn:beta_deriv}, it suffices to verify that
$$
\Om(q)\bigl[-2im_\ka\pm 2im_\ka C_+(q\overline m_\ka)\mp\tfrac{i}{\pi}m_\ka\beta_\ka\mp\tfrac{i}{\pi}q\|m_\ka\|^2_2\bigr]=2 \tfrac{\dl \beta_\ka}{\dl \bar{q}}.
$$
Using the identity
\begin{equation}\label{eqn:q_eqn}
q=(\Lc_q+\ka)m_\ka=-im_\ka'\mp q C_+(\bar{q}m_\ka)+\ka m_\ka
\end{equation}
and \eqref{Cid}, we may write
\begin{align*}
\Re\, &\bigl[\bar{q}(-2im_\ka\pm 2im_\ka C_+(q\overline m_\ka)\mp\tfrac{i}{\pi}m_\ka\beta_\ka\mp\tfrac{i}{\pi}q\|m_\ka\|_2^2)\bigr]\\
&=2\Re\,\bigl[-i\bigl(i\overline m_\ka'\mp\bar{q}C_-(q\overline m_\ka)+\ka\overline m_\ka\bigr)m_\ka\pm i\bar{q}m_\ka C_+(q\overline m_\ka)\mp \tfrac{i}{2\pi}\bar{q}m_\ka\beta_\ka\bigr]\\
&=2\Re\, \bigl[m_\ka\overline m_\ka'\pm i\bar{q}m_\ka C_-(q\overline m_\ka)\pm i\bar{q}m_\ka C_+(q\overline m_\ka)\mp \tfrac{i}{2\pi}\bar{q}m_\ka\beta_\ka\bigr]\\
&=(|m_\ka|^2)' \pm 2\Re\, \bigl[i\bar{q}m_\ka\bigl[q\overline m_\ka+(q\overline m_\ka)^{\wedge}(0)\bigr]-\tfrac{i}{2\pi}\bar{q}m_\ka\beta_\ka\bigr]=(|m_\ka|^2)'
\end{align*}
since $(q\overline m_\ka)^{\wedge}(0)=\frac1{2\pi}\beta_\ka$. It follows that
\begin{align*}
\Theta(q)\bigl[-2im_\ka&\pm 2im_\ka C_+(q\overline m_\ka)\mp\tfrac{i}{\pi} m_\ka\beta_\ka\mp\tfrac{i}{\pi}q\|m_\ka\|_2^2\bigr]\\
&=iq\tilde{\dd}^{-1}(|m_\ka|^2)'=iq\bigl[|m_\ka|^2-(|m_\ka|^2)^{\wedge}(0)\bigr]=iq|m_\ka|^2-\tfrac{i}{2\pi}q\|m_\ka\|^2_2.
\end{align*}
Therefore, using \eqref{Cid} again we obtain
\begin{align*}
\Om (q)&\bigl[-2im_\ka\pm 2im_\ka C_+(q\overline m_\ka)\mp\tfrac{i}{\pi}m_\ka\beta_\ka\mp\tfrac{i}{\pi}q\|m_\ka\|_2^2\bigr]\\
&=2m_\ka\mp 2m_\ka C_+(q\overline m_\ka)\pm \tfrac{1}{\pi} m_\ka\beta_\ka\pm \tfrac{1}{\pi}q\|m_\ka\|_2^2\mp2i C_+\bigl(iq|m_\ka|^2-\tfrac{i}{2\pi}q\|m_\ka\|^2_2\bigr)\\
&=2m_\ka\mp 2m_\ka C_+(q\overline m_\ka)\pm \tfrac{1}{\pi} m_\ka\beta_\ka\pm \tfrac{1}{\pi}q\|m_\ka\|^2_2 \\
&\quad \pm 2 C_+\bigl[m_\ka\bigl( C_+(q\overline m_\ka)+C_-(q\overline m_\ka)-(q\overline m_\ka)^{\wedge}(0)\bigr)\bigr] \mp \tfrac{1}{\pi}q\|m_\ka\|^2_2\\
&=2m_\ka\pm 2 C_+\bigl[m_\ka C_-(q\overline m_\ka)\bigr]=2 \tfrac{\dl \beta_\ka}{\dl \bar{q}},
\end{align*}
which proves \eqref{beta_symplectic_deriv}.

The claim \eqref{beta Lax pair} also follows from a lengthy but straightforward computation employing the identity \eqref{eqn:q_eqn}.  Note that this computation is simplified by the fact that $\Lc_q$ commutes with its resolvent and the constants $\beta_\kappa$ and $\|m_\ka\|^2_2$.  Finally, \eqref{one true peter} follows immediately from \eqref{beta_symplectic_deriv} and \eqref{eqn:Peter_op_ka}.
\end{proof}

To continue, we recall that $\beta(\ka,q)$ serves as a generating function for the infinite sequence of conservation laws $\mathcal{E}_n(q)=\lan q,\Lc_q^n q\ran$ in the sense that
$$
\beta(\ka,q)=\sum_{n\geq0}(-1)^n\ka^{-(n+1)}\mathcal{E}_n(q).
$$
In view of Proposition~\ref{lem:beta_results}, it is now a simple matter to derive the Hamiltonian flows corresponding to each $\mathcal{E}_n(q)$, as well as their attendant Lax pair formulations.

\begin{proposition}\label{P:En flows}
Throughout $B_{M_*}$ we have 
\begin{align}\label{E0 symplectic deriv}
\na_\om\mathcal{E}_0(q)= -2iq,
\end{align}
while for each $n\geq1$, we have
\begin{equation}\label{En symplectic deriv}
\begin{aligned}
\na_\om\mathcal{E}_n(q)&=-2i\Lc_q^nq\mp2i\!\!\sum_{j+\ell=n-1}\!\!(\Lc_q^jq)C_+(q\overline{\Lc_q^\ell q}) \qquad\qquad\qquad\qquad\qquad\qquad\quad \text{on $\R$,}\\
\na_\om\mathcal{E}_n(q)&=-2i\Lc_q^nq\mp2i\sum_{j+\ell=n-1}(\Lc_q^j q)\bigl[C_+(q\overline{\Lc_q^\ell q})-\tfrac1{2\pi}\mathcal{E}_\ell(q)\bigr]\pm\tfrac{i}{\pi}nq\E_{n-1}(q)\quad \text{on $\T$,}
\end{aligned}
\end{equation}
where $j,l\geq 0$. Moreover, these flows admit the Lax pair representation
\begin{equation}
\dd_t q=\na_\om\mathcal{E}_n(q) \qquad \iff \qquad \dd_t \Lc_q = \bigl[\mathcal{P}_n,\Lc_q\bigr]\label{En Lax pair}
\end{equation}
where $\mathcal{P}_0=-2i$ and for $n\geq 1$,
\begin{equation}\label{eqn:Peter_op_n}
\begin{aligned}
\mathcal{P}_n&:=-2i\Lc_q^n\mp 2i\!\!\!\!\!\sum_{j+\ell=n-1}(\Lc_q^jq) C_+\overline{\Lc_q^\ell q} \qquad\qquad\qquad\qquad\qquad\qquad \quad\text{on $\R$,}\\
\mathcal{P}_n&:=-2i\Lc_q^n\mp 2i\!\!\!\!\!\sum_{j+\ell=n-1}\bigl[(\Lc_q^jq) C_+\overline{\Lc_q^\ell q}-\tfrac{1}{2\pi}\E_\ell(q)\Lc_q^j\bigr]\pm\tfrac{i}{\pi}n\E_{n-1}(q)\quad \text{on $\T$.}
\end{aligned}
\end{equation}
Finally, for each $n\geq 0$, $\mathcal{P}_n$ enjoys the following special property:
\begin{equation}
\dd_t q=\na_\om\mathcal{E}_n(q) \qquad \iff \qquad \dd_t q=\mathcal{P}_n q.\label{one true peter for En}
\end{equation}
\end{proposition}

\begin{proof}
Once again, we present the details in the case of the circle. Expanding
\begin{align}\label{m expansion}
m_\ka =(\Lc_q+\ka)^{-1}q=\sum_{n\geq0}(-1)^n\ka^{-(n+1)}\Lc_q^nq
\end{align}
and substituting this into \eqref{beta_symplectic_deriv}, we find
\begin{align*}
\na_\om\beta_\ka
&=-2im_\ka\pm 2im_\ka C_+(q\overline{m}_\ka)\mp\tfrac{i}{\pi} m_\ka \beta_\ka\mp\tfrac{i}{\pi}q\| m_\ka\|_2^2\\
&=\sum_{n\geq0}2i\tfrac{(-1)^{n+1}}{\ka^{n+1}}\Lc_q^nq\pm2i\sum_{n,m\geq 0}\tfrac{(-1)^{n+m}}{\ka^{n+m+2}}(\Lc_q^nq)\bigl[C_+(q\overline{\Lc_q^mq})-\tfrac{1}{2\pi}\mathcal{E}_m(q)\bigr]\\
&\quad\pm\tfrac{i}{\pi} q\sum_{n,m\geq 0}\tfrac{(-1)^{n+m+1}}{\ka^{n+m+2}}\lan \Lc_q^nq,\Lc_q^mq\ran\\
&=-\tfrac{2i}{\ka} q - 2i\sum_{n\geq1}\tfrac{(-1)^{n}}{\ka^{n+1}}\Bigl\{\Lc_q^nq\pm\sum_{j+\ell=n-1}\!\!(\Lc_q^jq)\bigl[C_+(q\overline{\Lc_q^\ell q})-\tfrac{1}{2\pi}\mathcal{E}_\ell(q)\bigr]\\
&\qquad\qquad\qquad\qquad \qquad\qquad\qquad \qquad \qquad - \tfrac{1}{2\pi}q\,\lan \Lc_q^jq,\Lc_q^\ell q\ran\Bigr\}.
\end{align*}
Matching the coefficients in the expansion above with those in the expansion
\begin{equation*}
\quad\na_\om\beta_\ka=\sum_{n\geq0}\tfrac{(-1)^{n}}{\ka^{n+1}}\na_\om\mathcal{E}_n
\end{equation*}
and using the self-adjointness of $\Lc_q$, we arrive at \eqref{E0 symplectic deriv} and \eqref{En symplectic deriv}.

An analogous argument using \eqref{eqn:Peter_op_ka} and \eqref{one true peter} together with \eqref{m expansion} yields the Lax pair representation \eqref{En Lax pair} with the Peter operators defined in \eqref{eqn:Peter_op_n}. Finally, \eqref{one true peter for En} follows immediately from \eqref{En symplectic deriv} and \eqref{eqn:Peter_op_n}.
\end{proof}

Using Proposition~\ref{P:En flows} we see that  
\begin{align*}
\na_\om\mathcal{E}_0(q)= -2iq, \quad \na_\om\mathcal{E}_1(q)=-2q', \qtq{and} \na_\om\mathcal{E}_2(q)=2iq''\pm4qC_+(|q|^2)'
	\quad\text{on $\R$.}
\end{align*}
By contrast, the analogous formulas take the form
\begin{align}
\na_\om\mathcal{E}_0(q)&= -2iq,\quad\na_\om\mathcal{E}_1(q)=-2q'\pm\tfrac{2i}{\pi}q\mathcal{E}_0(q), \quad\text{and}\label{eqn:E0 and E1}\\
\na_\om\mathcal{E}_2(q)&=2iq''\pm4qC_+(|q|^2)'\pm \tfrac1{\pi}q'\mathcal{E}_0(q)\pm\tfrac{3i}{\pi}q\mathcal{E}_1(q) \quad\text{on $\T$.} \label{eqn: E2}
\end{align}

This motivates the following definitions for the mass, momentum, and Hamiltonian, respectively:
\begin{equation}\label{Hamiltonians on R}
M(q)=\mathcal{E}_0(q), \quad P(q)=-\tfrac12\mathcal{E}_1(q), \qtq{and}  H(q)=\tfrac12\mathcal{E}_2(q) \qtq{on} \R,
\end{equation}
while
\begin{equation}\label{Hamiltonians on T}
\begin{alignedat}{2}
M(q)&=\mathcal{E}_0(q), \qquad  P(q)=-\tfrac12\mathcal{E}_1(q)\mp\tfrac1{4\pi}\mathcal{E}_0(q)^2, \quad&&\text{and}\\
H(q)&=\tfrac12\mathcal{E}_2(q)\pm\tfrac{3}{4\pi}\mathcal{E}_0(q)\mathcal{E}_1(q)+\tfrac{1}{4\pi^2}\mathcal{E}_0(q)^3 &&\text{on}\quad \T. 
\end{alignedat}
\end{equation}

With these definitions, $M(q)$ generates phase rotation (at double speed) in both geometries.  Similarly, $P(q)$ was chosen so that it generates spatial translation in both geometries, which is to say $\na_\om P(q) = -q'$. Using \eqref{eqn:E0 and E1} and \eqref{eqn: E2} it is easy to  verify that
\begin{equation}\label{5:09}
\begin{aligned}
\na_\om H(q) &= iq''\pm 2qC_+(|q|^2)'\qquad \qquad \quad \quad\text{on $\R$}\\
\na_\om H(q) &= iq''\pm 2qC_+(|q|^2)'\mp\tfrac{1}{\pi}\mathcal E_0(q) q' \quad\text{on $\T$}.
\end{aligned}
\end{equation}
Thus, in the line setting, the Hamiltonian $H(q)$ defined in \eqref{Hamiltonians on R} generates the flow \eqref{CCMR}, while in the circle setting, the Hamiltonian $H(q)$ defined in \eqref{Hamiltonians on T} generates the flow \eqref{CCMT}.

Note that there is some freedom in the choice of the Hamiltonian $H(q)$ on $\T$; we opted for the version above in order to obtain an equation that (i) is as close as possible to \eqref{CCMR} and (ii) is well-posed for data in $L^2_+$.

It remains to verify the equations \eqref{eqn:H(q) R} and \eqref{eqn:H(q) T}. As already recorded in \cite{gerard2024calogero}, on the real line we have
$$
P(q) =-\tfrac12\|q\|_{\dot H^{1/2}}^2\pm \tfrac14\|q\|_{L^4}^4 \qtq{and} H(q)=\tfrac12\|q'\mp i qC_+(|q|^2)\|_{L^2}^2.
$$
This is also easily derived from the computations in the proof of Lemma~\ref{L:P and H}, whose role is to obtain analogous expressions for $P(q)$ and $H(q)$ in the case of the circle:

\begin{lemma}\label{L:P and H}
On $\T$ we have 
\begin{align*}
P(q)&=-\tfrac{1}{2}\|q\|_{\dot{H}^{1/2}}^2\pm \tfrac14\|q\|_{L^4}^4\mp\tfrac1{8\pi}\|q\|_{L^2}^4,\\
H(q)&=\tfrac12\|q'\mp i qC_+(|q|^2)\|_{L^2}^2\pm \tfrac{3}{4\pi} \|q\|_{L^2}^2\|q\|_{\dot{H}^{1/2}}^2-\tfrac{3}{8\pi}\|q\|_{L^2}^2\|q\|_{L^4}^4+\tfrac{1}{16\pi^2}\|q\|_{L^2}^6\notag\\
&= \tfrac12\|q\|_{\dot{H}^{1}}^2\mp\tfrac12\| C_+(|q|^2)\|_{\dot{H}^{1/2}}^2\mp\tfrac14\|q^2\|_{\dot{H}^{1/2}}^2\pm \tfrac1{2\pi}\|q\|_{L^2}^2\|q\|_{\dot{H}^{1/2}}^2\nonumber\\
&\quad+\tfrac{1}{6}\|q\|_{L^6}^6-\tfrac1{4\pi}\|q\|_{L^4}^4\|q\|_{L^2}^2+\tfrac{1}{12\pi^2}\|q\|_{L^2}^6.
\end{align*}
\end{lemma}

\begin{proof}
To establish the expressions for $P(q)$ and $H(q)$ we first calculate
\begin{align*}
\mathcal{E}_1(q)&=\lan q,\Lc_qq\ran=\|q\|_{\dot{H}^{1/2}}^2\mp\lan |q|^2, C_+(|q|^2)\ran.
\end{align*}
Employing \eqref{Cid}, we may write
\begin{align*}
\lan |q|^2, C_+(|q|^2)\ran&=\lan |q|^2,|q|^2-C_-(|q|^2)+\tfrac1{2\pi}\mathcal{E}_0(q)\ran\\
&=\|q\|_{L^4}^4-\lan |q|^2,C_-(|q|^2)\ran+\tfrac1{2\pi}\mathcal{E}_0^2(q).
\end{align*}
As $\lan |q|^2,C_-(|q|^2)\ran=\overline{\lan |q|^2,C_+(|q|^2)\ran}=\lan |q|^2,C_+(|q|^2)\ran$, we find
\begin{equation}\label{eqn:norm_Cp_q^2}
\lan C_+ (|q|^2),C_+(|q|^2)\ran=\lan |q|^2,C_+(|q|^2)\ran=\tfrac12\|q\|_{L^4}^4+\tfrac1{4\pi}\|q\|_{L^2}^4.
\end{equation}
Thus,
\begin{align}\label{7:50}
\mathcal{E}_1(q)=\|q\|_{\dot{H}^{1/2}}^2\mp\tfrac12\|q\|_{L^4}^4\mp\tfrac1{4\pi}\|q\|_{L^2}^4
\end{align}
and the claimed expression for $P(q)$ follows.

As 
\begin{align}\label{5:08}
\mathcal{E}_2(q)&=\lan \Lc_qq,\Lc_qq\ran=\|q'\mp i qC_+(|q|^2)\|_{L^2}^2,
\end{align}
employing \eqref{7:50} we obtain
\begin{align*}
H(q)&=\tfrac12\mathcal{E}_2(q)\pm\tfrac{3}{4\pi}\mathcal{E}_0(q)\mathcal{E}_1(q)+\tfrac{1}{4\pi^2}\mathcal{E}_0(q)^3\\
&=\tfrac12\|q'\mp i qC_+(|q|^2)\|_{L^2}^2\pm \tfrac{3}{4\pi} \|q\|_{L^2}^2\|q\|_{\dot{H}^{1/2}}^2-\tfrac{3}{8\pi}\|q\|_{L^2}^2\|q\|_{L^4}^4+\tfrac{1}{16\pi^2}\|q\|_{L^2}^6.
\end{align*}

To obtain the detailed representation of $H(q)$ in the statement of this lemma, we have to further expand $\mathcal E_2(q)$.  We have
\begin{align}\label{eqn:E_2_expansion}
\mathcal{E}_2(q)&=\lan \Lc_qq,\Lc_qq\ran=\|q'\|_{L^2}^2\pm 2\Re\,\lan iq',q C_+(|q|^2)\ran+\|q C_+(|q|^2)\|_{L^2}^2.
\end{align}
First we study
\begin{align*}
2\Re\,\lan iq',q C_+(|q|^2)\ran&=2\Re\,\lan i\bar{q}q', C_+(|q|^2)\ran\\
&=2\Re\,\lan i\bigl[(|q|^2)'-q\overline{q'}\,\bigr], C_+(|q|^2)\ran.
\end{align*}
Averaging the two lines above, we obtain
\begin{align}\label{7:46}
2\Re\,\lan iq',q C_+(|q|^2)\ran&=\Re\,\lan i(|q|^2)', C_+(|q|^2)\ran+\Re\,\lan i(\bar{q}q'-q\overline{q'}), C_+(|q|^2)\ran.
\end{align}
The first term is
\begin{align}\label{7:47}
\Re\,\lan i(|q|^2)', C_+(|q|^2)\ran=-\| C_+(|q|^2)\|_{\dot{H}^{1/2}}^2.
\end{align}
For the second term, we use that $ C_+(|q|^2)=\overline{C_-(|q|^2)}$ to write
\begin{align*}
\Re\,\lan i(\bar{q}q'-q\overline{q'}), C_+(|q|^2)\ran
&=-2\Re\,\lan\Im(\bar{q}q'), C_+(|q|^2)\ran\\
&=-2\Re\,\lan\Im(\bar{q}q'),C_-(|q|^2)\ran.
\end{align*}
Averaging again the last two lines and invoking \eqref{Cid} we obtain
\begin{align}\label{7:48}
\Re\, \lan i(\bar{q}q'-q\overline{q'}), C_+(|q|^2)\ran\notag
&=-\Re\,\lan \Im(\bar{q}q'),|q|^2+\tfrac1{2\pi}\mathcal{E}_0(q)\ran\notag\\
&=\Im\,\lan\bar{q}q',|q|^2\ran-\tfrac1{2\pi}\mathcal{E}_0(q)\Im\lan q,q'\ran\notag\\
&=\tfrac12 \Im\,\lan (q^2)',q^2\ran-\tfrac1{2\pi}\mathcal{E}_0(q)\|q\|_{\dot H^{1/2}}^2\notag\\
&=-\tfrac12\|q^2\|_{\dot{H}^{1/2}}^2-\tfrac1{2\pi}\|q\|_{L^2}^2\|q\|_{\dot{H}^{1/2}}^2.
\end{align}
Collecting \eqref{7:46}, \eqref{7:47}, and \eqref{7:48}, we obtain
\begin{align}\label{7:49}
2\Re\,\lan iq',q C_+(|q|^2)\ran=-\| C_+(|q|^2)\|_{\dot{H}^{1/2}}^2-\tfrac12\|q^2\|_{\dot{H}^{1/2}}^2-\tfrac1{2\pi}\|q\|_{L^2}^2\|q\|_{\dot{H}^{1/2}}^2.
\end{align}

Next, employing \eqref{Cid} we compute
\begin{align*}
\|q C_+(|q|^2)\|^2_{L^2}
&=\int_\T |q|^2 C_+(|q|^2)C_-(|q|^2)\,dx\\
&=\int_\T  \bigl[C_+(|q|^2)+C_-(|q|^2) -\tfrac1{2\pi} \mathcal{E}_0(q)\bigr]C_+(|q|^2)C_-(|q|^2)\,dx\\
&=\tfrac{1}{3}\int_\T\bigl[C_+(|q|^2)+C_-(|q|^2)\bigr]^3-\bigl[ C_+(|q|^2)\bigr]^3-\bigl[C_-(|q|^2)\bigr]^3\,dx\\
&\quad-\tfrac1{4\pi}\|q\|_{L^2}^2\|q\|_{L^4}^4-\tfrac1{8\pi^2}\|q\|_{L^2}^6,
\end{align*}
where we used \eqref{eqn:norm_Cp_q^2} in the final line.  Using again \eqref{Cid} we find
\begin{align*}
\int_\T\bigl[C_+(|q|^2)+C_-(|q|^2)\bigr]^3\,dx
&=\int_\T\bigl[|q|^2+\tfrac1{2\pi}\mathcal{E}_0(q)\bigr]^3\,dx\\
&=\|q\|_{L^6}^6+\tfrac3{2\pi}\|q\|_{L^4}^4\|q\|_{L^2}^2+\tfrac{1}{\pi^2}\|q\|_{L^2}^6.
\end{align*}
Using Fourier support considerations, we obtain
\begin{align*}
\int_\T\bigl[ C_+(|q|^2)\bigr]^3\,dx&=2\pi\bigl[\tfrac{1}{2\pi}\mathcal{E}_0(q)\bigr]^3=\tfrac{1}{4\pi^2} \|q\|_{L^2}^6
\end{align*}
and similarly,
\begin{align*}
\int_\T\bigl[ C_-(|q|^2)\bigr]^3\,dx&=2\pi\bigl[\tfrac{1}{2\pi}\mathcal{E}_0(q)\bigr]^3=\tfrac{1}{4\pi^2} \|q\|_{L^2}^6.
\end{align*}
We deduce that
\begin{align}
\|q C_+(|q|^2)\|^2_{L^2}
&=\tfrac{1}{3}\|q\|_{L^6}^6+\tfrac1{4\pi}\|q\|_{L^4}^4\|q\|_{L^2}^2+\tfrac{1}{24\pi^2}\|q\|_{L^2}^6.\label{eqn:L6_expression}
\end{align}

Combining \eqref{eqn:E_2_expansion} with \eqref{7:49} and \eqref{eqn:L6_expression}, we obtain 
\begin{align}
\mathcal{E}_2(q)&=\|q\|_{\dot{H}^{1}}^2\mp\bigl[\| C_+(|q|^2)\|_{\dot{H}^{1/2}}^2+\tfrac12\|q^2\|_{\dot{H}^{1/2}}^2+\tfrac1{2\pi}\|q\|_{L^2}^2\|q\|_{\dot{H}^{1/2}}^2\bigr]\nonumber\\
&\quad+\tfrac{1}{3}\|q\|_{L^6}^6+\tfrac1{4\pi}\|q\|_{L^4}^4\|q\|_{L^2}^2+\tfrac{1}{24\pi^2}\|q\|_{L^2}^6.\label{eqn:E2}
\end{align}

Inserting this into \eqref{5:08} we find
\begin{align*}
H(q)&= \tfrac12\|q\|_{\dot{H}^{1}}^2\mp\tfrac12\| C_+(|q|^2)\|_{\dot{H}^{1/2}}^2\mp\tfrac14\|q^2\|_{\dot{H}^{1/2}}^2\pm \tfrac1{2\pi}\|q\|_{L^2}^2\|q\|_{\dot{H}^{1/2}}^2\\
&\quad+\tfrac{1}{6}\|q\|_{L^6}^6-\tfrac1{4\pi}\|q\|_{L^4}^4\|q\|_{L^2}^2+\tfrac{1}{12\pi^2}\|q\|_{L^2}^6,
\end{align*}
as desired.
\end{proof}

We end this section with a comment on the coercivity of the Hamiltonian $H(q)$ on $\T$. Note that in both the focusing and defocusing cases, the Cauchy--Schwarz inequality applied to $\E_1$ yields
\begin{align*}
H(q)&\geq \tfrac{1}{2}\E_2(q)-\tfrac{3}{4\pi}\E_0(q)\sqrt{\E_0(q)\E_2(q)}+\tfrac{1}{4\pi^2}\E_0(q)^3= \tfrac{1}{2}\E_0(q)^3\bigl(X^2-\tfrac3{2\pi}X+\tfrac1{2\pi^2}\bigr)
\end{align*}
with $X=\sqrt{\E_2(q)/\E_0(q)^3}$. As the quadratic attains its minimum at $X=\frac3{4\pi}$, we see that
$$H(q)\geq-\tfrac1{32\pi^2}\E_0^3(q).$$
This lower bound is sharp in the defocusing case; concretely, equality holds for the sequence
$$
q_m(x) = \sqrt{2m}\, e^{imx} \qtq{with} m\in \N ,
$$
which has mass $M(q_m)=4\pi m$ diverging to infinity.   This sequence is easy to analyze because $q_m$ is an eigenvector of the Lax operator with eigenvalue $3m$.


\section{The conserved quantities commute}\label{sec:commute}
In this section we show that for all $\ka, \vk\geq 1$, the Hamiltonians $\beta_\ka(q) =\beta(\ka,q)$ and $\beta_\vk(q)=\beta(\vk, q)$ Poisson commute on $B_{M_*}$.  Consequently, their attendant Hamiltonian flows commute.  As $\beta_\ka(q)$ is a generating function for the quantities $\mathcal{E}_n(q)$, this also shows that these quantities are conserved along the flow generated by $\beta_\vk(q)$ for all $\vk\geq 1$.  A further application of this argument yields that all quantities $\mathcal{E}_n(q)$ are conserved along the flows generated by any $\mathcal{E}_m(q)$ with $m\geq 0$.  In particular, all $\mathcal{E}_n(q)$ are conserved along the flows \eqref{CCMR} and \eqref{CCMT}.

\begin{proposition}\label{P:beta comm}
For all $\ka, \vk\geq 1$ we have
$$
\{\beta_\vk(q),\beta_\ka(q)\}=0
$$
on $B_{M_*}$. In particular, 
$$
\{\beta_\vk(q),\mathcal E_n(q)\}=0 \qtq{and} \{\mathcal E_m(q),\mathcal E_n(q)\}=0
$$
for all $\vk\geq 1$ and $m,n\geq 0$.
\end{proposition}

\begin{proof}
Recalling the definition \eqref{PBJ} of the Poisson bracket, \eqref{J}, and \eqref{symplectic gradient}, we may write
$$
\{\beta_\vk(q),\beta_\ka(q)\}=-\om_q(\na_\om \beta_\vk(q),\na_\om \beta_\ka(q))=- 2\Re\, \lan \na_\om \beta_\vk(q),\tfrac{\dl \beta_\ka}{\dl \bar{q}}\bigr|_q\ran.
$$

In what follows, we again focus attention to the circle setting since it is the more complicated case.  Employing \eqref{eqn:beta_deriv} and \eqref{beta_symplectic_deriv}, we may compute
\begin{align}
\{\beta_\vk(q),\beta_\ka(q)\}&=4\Im\,\lan m_\vk,m_\ka\ran \pm 4\Im\,\lan m_\vk, m_\ka C_-(q\overline m_\ka)\ran\nonumber\\
&\quad\mp 4\Im\,\lan m_\vk C_+(q\overline{m}_\vk),m_\ka\ran-4\Im\,\lan m_\vk C_+(q\overline{m}_\vk),m_\ka C_-(q\overline m_\ka)\ran \nonumber\\
&\quad\pm\tfrac{2}{\pi}\Im\,\beta_\vk\lan m_\vk,m_\ka\ran+\tfrac{2}{\pi}\Im\,\beta_\vk\lan m_\vk,m_\ka C_-(q\overline{m}_\ka)\ran \nonumber\\
&\quad\pm\tfrac{2}{\pi} \Im\, \beta_\ka\| m_\vk\|_{2}^2+\tfrac{2}{\pi}\| m_\vk\|_{2}^2\Im\, \lan q, m_\ka  C_-(q\overline{m}_\ka)\ran. \label{eqn:poisson_bracket}
\end{align}
Using that $\Lc_q$ is self-adjoint, we see that $\Im\, \beta_\ka=0$ and $\Im\,\lan m_\vk,m_\ka\ran=0$.  Also,
$$
\Im\, \lan q, m_\ka  C_-(q\overline{m}_\ka)\ran = \Im\, \lan q\overline{m}_\ka, C_-(q\overline{m}_\ka)\ran =\Im\,\|C_-(q\overline{m}_\ka)\|_2^2=0.
$$
Thus, \eqref{eqn:poisson_bracket} simplifies to
\begin{align}
\{\beta_\vk(q),\beta_\ka(q)\}&= \pm 4\Im\,\lan m_\vk, m_\ka C_-(q\overline m_\ka)\ran\mp 4\Im\,\lan m_\vk C_+(q\overline{m}_\vk),m_\ka\ran\nonumber\\
&\quad-4\Im\,\lan m_\vk C_+(q\overline{m}_\vk),m_\ka C_-(q\overline m_\ka)\ran+\tfrac{2}{\pi}\beta_\vk\Im\,\lan m_\vk,m_\ka C_-(q\overline{m}_\ka)\ran. \label{eqn:poisson_bracket'}
\end{align}

Substituting $q=(\Lc_q+\vk) m_\vk$ and $\bar q =\overline{(\Lc_q+\vk) m_\vk} $, we may write
\begin{align*}
\pm4\Im\,\lan &m_\vk, m_\ka C_-(q\overline m_\ka)\ran \mp4\Im\,\lan m_\vk C_+(q\overline m_\vk),m_\ka\ran \\
&=\pm4\Im\,\lan  m_\vk \overline m_\ka,  C_-\bigl[-i m_\vk '\overline m_\ka\mp q\overline m_\ka C_+(\bar{q} m_\vk )+\vk m_\vk \overline m_\ka\bigr]\ran \\
&\quad\mp4\Im\,\lan  m_\vk \overline m_\ka, C_-\bigl[i m_\vk \overline m_\ka '\mp \bar{q}m_\vk C_-(q \overline m_\ka )+\ka m_\vk \overline m_\ka\bigr]\ran\\
&= \mp4\Im\,  \lan  m_\vk \overline m_\ka,  i C_-\bigl[m_\vk \overline m_\ka\bigr]'\ran \pm4 (\vk-\ka)\Im\, \|C_-( m_\vk \overline m_\ka)\|_2^2\\
&\quad -4\Im\,\lan  m_\vk \overline m_\ka,  C_-\bigl[q\overline m_\ka C_+(\bar{q} m_\vk )\bigr]\ran+4 \Im\,\lan  m_\vk \overline m_\ka, C_-\bigl[\bar{q}m_\vk C_-(q \overline m_\ka )\bigr]\ran\\
&=-4\Im\,\lan  m_\vk \overline m_\ka,  C_-\bigl[q\overline m_\ka C_+(\bar{q} m_\vk )\bigr]\ran+4 \Im\,\lan  m_\vk \overline m_\ka, C_-\bigl[\bar{q}m_\vk C_-(q \overline m_\ka )\bigr]\ran.
\end{align*}
To continue, we use \eqref{Cid} and Fourier support considerations to write
\begin{align*}
&-4\Im\,\lan  m_\vk \overline m_\ka,  C_-\bigl[q\overline m_\ka C_+(\bar{q} m_\vk )\bigr]\ran+4 \Im\,\lan  m_\vk \overline m_\ka, C_-\bigl[\bar{q}m_\vk C_-(q \overline m_\ka )\bigr]\ran\\
&=-4\Im\,\lan  m_\vk \overline m_\ka, C_-\bigl[C_+(q\overline m_\ka )C_+(\bar{q} m_\vk )\bigr]\ran -4\Im\,\lan  m_\vk \overline m_\ka ,  C_-\bigl[C_-(q\overline m_\ka )C_+(\bar{q} m_\vk )\bigr]\ran\\
&\quad +\tfrac2{\pi}\beta_\ka \Im\,\lan  m_\vk \overline m_\ka ,  C_-\bigl[C_+(\bar{q} m_\vk )\bigr]\ran+4 \Im\,\lan  m_\vk \overline m_\ka , C_-\bigl[C_+(\bar{q}m_\vk) C_-(q \overline m_\ka  )\bigr]\ran\\
&\quad+4 \Im\,\lan  m_\vk \overline m_\ka , C_-(\bar{q}m_\vk) C_-(q \overline m_\ka  )\ran-\tfrac{2}{\pi} \beta_\vk\Im\,\lan  m_\vk \overline m_\ka , C_-(q \overline m_\ka  )\ran\\
&=- \tfrac1{\pi^2}\beta_\ka\beta_\vk\Im\,\lan  m_\vk,m_\ka\ran+\tfrac1{\pi^2}\beta_\ka\beta_\vk\Im\,\lan  m_\vk,m_\ka\ran \\
&\quad+4 \Im\,\lan  m_\vk \overline m_\ka , C_-(\bar{q}m_\vk) C_-(q \overline m_\ka  )\ran-\tfrac{2}{\pi} \beta_\vk\Im\,\lan  m_\vk \overline m_\ka , C_-(q \overline m_\ka  )\ran\\
&=4\Im\,\lan m_\vk C_+(q\overline{m_\vk}),m_\ka C_-(q\overline m_\ka )\ran -\tfrac{2}{\pi}\beta_\vk\Im\,\lan m_\vk,m_\ka C_-(q\overline{m}_\ka)\ran.
\end{align*}
Inserting this into \eqref{eqn:poisson_bracket'}, we find
\begin{align*}
\{\beta_\vk(q),\beta_\ka(q)\}=0,
\end{align*}
as desired.
\end{proof}

\section{Global well-posedness of the $H_\ka$ flow}\label{sec:GWP_kappa}
The expansion 
$$
\beta(\ka,q)=\sum_{n\geq0}(-1)^n\ka^{-(n+1)}\mathcal{E}_n(q),
$$
motivates the introduction of the regularized energies
\begin{equation}\label{11:47}
\E_1^\ka(q)=-\ka^2\beta(\ka,q) + \ka\E_0(q) \qtq{and}  \E_2^{\ka}(q)=\ka^3\beta(\ka,q)-\ka^2\E_0(q)+\ka\E_1(q).
\end{equation}

This was the strategy used in previous applications of the method of commuting flows to select the regularized Hamiltonians $H_\ka$.  In the line setting, this approach still works and for each $\kappa\geq 1$ we may define the regularized Hamiltonians
\begin{align}
H_{\ka}(q)&:=\tfrac12 \E_2^{\ka}(q)  \qquad \text{on $\R$}.  \label{eqn:H_ka R}
\end{align}

The naive modification of \eqref{Hamiltonians on T} in the circle setting would be to replace $\E_1(q)$ and $\E_2(q)$ by their regularized versions introduced in \eqref{11:47}. However, the resulting regularized Hamiltonian flows are not well-posed on $L^2_+(\T)$. Instead, we employ the regularized Hamiltonians
\begin{align}
H_{\ka}(q)&:=\tfrac12 \E_2^{\ka}(q)\pm\tfrac{3}{4\pi}\E_0(q)\E_1^{\ka}(q)+\tfrac{1}{4\pi^2}\E_0^3(q)\mp\tfrac1{4\pi} \ka\E_0(q)\| \Lc_qm_\kappa\|_{L^2}^2 \quad \text{on $\T$}.\label{eqn:H_ka T}
\end{align}

The final term in \eqref{eqn:H_ka T} has no analogue in \eqref{Hamiltonians on T}.  In fact, for smooth $q$, this term drops out in the $\ka\to\infty$ limit.  Nevertheless, it meaningfully affects the dynamics and its inclusion is necessary to ensure convergence of the $H_\ka(q)$ flow to the $H(q)$ flow at low regularity; see Section \ref{sec:convergence}. For the method of commuting flows, it is also essential that any renormalization maintains the commuting flow property. Our renormalization is based on a commuting Hamiltonian; see \eqref{11:51}.

The goal of this section is to prove global well-posedness on $L^2$ for the flows generated by $H_\ka(q)$:

\begin{theorem}\label{thm:GWP_ka}
Fix $\ka\geq1$. The initial-value problem
\begin{align}\label{IV Hk}
\dd_t q=\na_\om H_\ka(q),\qquad q(0)=q_0
\end{align}
is globally well-posed on $B_{M_*}$. Moreover, for any $M<M_*$, if the set $Q\subset B_M$ is $L^2$-equicontinuous, then the set of orbits
$$
Q_{\ast}:=\bigl\{e^{t\na_\om H_\ka}q_0:\,q_0\in Q,\, t\in \R,\,\ka\geq 1 \bigr\}
$$
is also $L^2$-bounded and equicontinuous.
\end{theorem}

We will construct solutions to \eqref{IV Hk} using the Picard theorem, which requires that the forcing term be a Lipschitz function of $q$.  The proof of this fact relies on the estimates contained in the next lemma:

\begin{lemma}\label{prop:H1_lipschitz}
Fix $M< M_\ast$ and $\ka\geq 1$. For all $p,q\in B_M$ and $f,g\in L^2_+$, we have
\begin{align}\label{8:49}
\|(\Lc_q+\ka)^{-1}f-(\Lc_p+\ka)^{-1}g\|_{H^1}\lesssim_M \|f-g\|_{L^2}+\|f\|_{L^2}\|q-p\|_{L^2}.
\end{align}
In particular,
\begin{align}
\|m(\ka,q)-m(\ka,p)\|_{H^1}&\lesssim_M\|q-p\|_{L^2}, \label{eqn:m_lipschitz} \\
\|(\Lc_q+\ka)^{-1} m(\ka,q)-(\Lc_p+\ka)^{-1} m(\ka,p)\|_{H^1}&\lesssim_M\|q-p\|_{L^2}, \label{eqn:H1_lipschitz3}\\
\|\Lc_q m(\ka,q)-\Lc_p m(\ka,p)\|_{L^2}&\lesssim_{M, \ka}\|q-p\|_{L^2}. \label{eqn:L2_lipschitz} 
\end{align}
\end{lemma}

\begin{proof}
Using the resolvent identity, we may write
\begin{equation*}
(\Lc_q+\ka)^{-1}f-(\Lc_p+\ka)^{-1}g=(\Lc_q+\ka)^{-1}(\Lc_p-\Lc_q)(\Lc_p+\ka)^{-1}f+(\Lc_p+\ka)^{-1}(f-g).
\end{equation*}

By Theorem~\ref{T:equiv} we see that
\begin{align}\label{8:51}
\|(\Lc_p+\ka)^{-1}(f-g)\|_{H^1}\lesssim_M\|f-g\|_{L^2}.
\end{align}

Recall that \eqref{E:QHS} for the operator $T$ yields
\begin{align}\label{8:50}
\|uC_+\bar v (\mathcal L_0+\ka)^{-1}\|_{\op}\lesssim \|u\|_{L^2} \|v\|_{L^2},
\end{align}
uniformly for $u,v\in L^2_+$. Employing this together with Theorem~\ref{T:equiv}, we may bound
\begin{align*}
\|(\Lc_q+\ka)^{-1}&(\Lc_p-\Lc_q)(\Lc_p+\ka)^{-1}f\|_{H^1}\\
&\lesssim_M \|(p-q) C_+\bigl[\bar{p}(\Lc_p+\ka)^{-1}f\bigr]\|_{L^2}+\|q C_+\bigl[(\bar{p}-\bar{q})\,(\Lc_p+\ka)^{-1}f\bigr]\|_{L^2}\\
&\lesssim_M \|p-q\|_{L^2}\|p\|_{L^2}\|(\mathcal L_0+\ka)(\Lc_p+\ka)^{-1}f\|_{L^2}\\
&\quad +\|q\|_{L^2}\|p-q\|_{L^2}\|(\mathcal L_0+\ka)(\Lc_p+\ka)^{-1}f\|_{L^2}\\
&\lesssim_M \|p-q\|_{L^2}\|f\|_{L^2}.
\end{align*}
Combining this with \eqref{8:51} yields \eqref{8:49}.

Applying \eqref{8:49} with $f=q$ and $g=p$ yields \eqref{eqn:m_lipschitz}.  Substituting $f=m(\ka,q)$ and $g=m(\ka,p)$ in \eqref{8:49} yields \eqref{eqn:H1_lipschitz3}.

Finally, using \eqref{eqn:m_lipschitz} and \eqref{8:50}, we may bound
\begin{align*}
\|\Lc_q &m(\ka,q)-\Lc_p m(\ka,p)\|_{L^2}\\
&\lesssim \|m(\ka,q) - m(\ka,p)\|_{H^1} + \|(q-p)C_+ \bar q m(\ka, q)\|_{L^2} \\
&\quad + \|pC_+ (\bar q-\bar p) m(\ka, q)\|_{L^2} +\|pC_+ \bar p [m(\ka, q)-m(\ka,p)]\|_{L^2} \\
&\lesssim_M \|q-p\|_{L^2} + \|q-p\|_{L^2}(\|p\|_{L^2}+\|q\|_{L^2})\|(\mathcal L_0+\ka)m(\ka,q)\|_{L^2}\\
&\quad +\|p\|_{L^2}^2\|\mathcal (L_0+\ka)[m(\ka,q)-m(\ka,p)]\|_{L^2}\\
&\lesssim_{M,\ka}  \|q-p\|_{L^2},
\end{align*}
which yields \eqref{eqn:L2_lipschitz}.
\end{proof}

\begin{proof}[Proof of Theorem~\ref{thm:GWP_ka}]
We present the details in the case of the circle, which is the more challenging of the two geometries. We begin by computing the symplectic gradient of $H_\ka$.  To keep formulas within the margins, throughout the proof we use the shorthand $\beta$ and $m$ for $\beta(\ka,q)$ and $m(\ka,q)$, respectively.

Recalling \eqref{beta_symplectic_deriv}, \eqref{eqn:E0 and E1}, and \eqref{11:47}, we find
\begin{equation}\label{symplectic deriv reg E}
\begin{aligned}
\na_\om\E_1^\ka&=-\ka^2\bigl[-2im\pm 2im C_+(q\overline{m})\mp\tfrac{i}{\pi} m \beta\mp\tfrac{i}{\pi}q\|m \|^2_{L^2}\bigr] -2i\ka q\\
\na_\om\E_2^\ka&=\ka^3\bigl[-2im\pm 2im C_+(q\overline{m})\mp\tfrac{i}{\pi} m \beta\mp\tfrac{i}{\pi}q\|m \|^2_{L^2}\bigr] \\
&\quad+2i\ka^2 q +\ka\bigl[ -2q'\pm\tfrac{2i}{\pi}q\mathcal{E}_0\bigr].
\end{aligned}
\end{equation}
Observing that
\begin{align}\label{11:51}
\|\Lc_q m\|_{L^2}^2 = \langle q, \Lc_q^2 (\Lc_q+\ka)^{-2}q\rangle &=  \langle q, \bigl[1-2\ka (\Lc_q+\ka)^{-1} + \ka^2(\Lc_q+\ka)^{-2}\bigr] q\rangle\nonumber \\
&=\E_0 - 2\ka\beta-\ka^2 \tfrac{\partial\beta}{\partial\kappa},
\end{align}
we deduce that
\begin{align}
\na_\om\| \Lc_qm\|_{L^2}^2&=\na_\om\E_0 - 2\ka\na_\om\beta-\ka^2 \tfrac{\partial}{\partial\kappa}\na_\om\beta\nonumber\\
&=-2iq+4i\ka m\mp4i\ka m C_+(q\overline{m})\pm\tfrac{2i}{\pi}\ka m\beta\pm\tfrac{2i}{\pi}\ka q\|m\|^2_{L^2}\nonumber\\
&\quad-2i \ka^2(\Lc_q+\ka)^{-1}m\pm2i\ka^2[(\Lc_q+\ka)^{-1}m] C_+(q\overline{m})\nonumber\\
&\quad\pm2i\ka^2 m C_+[q\overline{(\Lc_q+\ka)^{-1}m}]\mp\tfrac{i}{\pi}\ka^2\beta (\Lc_q+\ka)^{-1}m\nonumber\\
&\quad\mp\tfrac{i}{\pi}\ka^2 m\|m\|_{L^2}^2\mp\tfrac{2i}{\pi}\ka^2 q\lan m,(\Lc_q+\ka)^{-1}m\ran.\label{eqn:na_om_Lqm}
\end{align}
Inserting all these into \eqref{eqn:H_ka T}, we find
\begin{align*}
\na_\om H_\ka&=(\ka^3\mp \tfrac3{2\pi}\ka^2\E_0)\bigl[-im\pm im C_+(q\overline{m})\mp \tfrac{i}{2\pi}m\beta \mp \tfrac{i}{2\pi}q\|m\|_{L^2}^2\bigr]\\
&\quad+i\ka^2q -\ka q' \mp\tfrac{2i}{\pi}\ka q\E_0 \pm \tfrac{3i}{2\pi}\ka^2\beta q-\tfrac{3i}{2\pi^2}\E_0^2q\pm \tfrac{i}{2\pi}\ka q\| \Lc_q m\|_{L^2}^2\\
&\quad\mp\tfrac{1}{2\pi}\ka\E_0\Bigl\{-iq+2i\ka m\mp2i\ka m C_+(q\overline{m})\pm\tfrac{i}{\pi}\ka m\beta\pm\tfrac{i}{\pi}\ka q\|m\|^2_{L^2}\nonumber\\
&\quad-i \ka^2[(\Lc_q+\ka)^{-1}m]\pm i\ka^2[(\Lc_q+\ka)^{-1}m] C_+(q\overline{m})\nonumber\\
&\quad\pm i\ka^2 m C_+[q\overline{(\Lc_q+\ka)^{-1}m}]\mp\tfrac{i}{2\pi}\ka^2\beta(\Lc_q+\ka)^{-1}m\nonumber\\
&\quad\mp\tfrac{i}{2\pi}\ka^2 m\|m\|_{L^2}^2\mp\tfrac{i}{\pi}\ka^2 q\lan m,(\Lc_q+\ka)^{-1}m\ran\Bigr\}.
\end{align*}

In this way, the Duhamel formula corresponding to \eqref{IV Hk} takes the form
\begin{equation}\label{eqn:duhamel}
q(t)=e^{it\ka^2-\ka t\dd_x}q_0+\int_0^t e^{i(t-s)\ka^2-\ka(t-s)\dd_x} \mathcal{N}(q(s))\,ds,
\end{equation}
where $\mathcal{N}(q)=\na_\om H_\ka(q)-i\ka^2 q+\ka q'$. 

An application of Lemma~\ref{prop:H1_lipschitz} and Theorem~\ref{T:equiv} shows that the mapping $\mathcal{N}: B_M \mapsto L^2_+$ is Lipschitz, with Lipschitz constant depending only on $M$ and $\ka$. The Picard theorem then yields local well-posedness for the initial-value problem \eqref{IV Hk} with the time of existence depending only on the $L^2_+$ norm of the initial data and $\ka$.  To extend the solution to a global-in-time solution, it suffices to observe that the $H_\ka$ flow conserves the $L^2$ norm.  This follows from  Proposition~\ref{P:beta comm} or can be deduced directly from \eqref{IV Hk}.

We now turn to the equicontinuity statement in Theorem~\ref{thm:GWP_ka}.  To this end, let $Q\subset B_M$ be $L^2$-equicontinuous. As the $H_\ka$ flow conserves the $L^2$ norm, it is immediate that $Q_{\ast}\subset B_M$.  In view of Theorem~\ref{T:equi-c}, to prove that $Q_{\ast}$ is $L^2$-equicontinuous it suffices to show that
$$
\lim_{\varkappa\to\infty}\ \sup_{q\in Q_\ast}\ \bigl\langle q, \tfrac{(\Lc_q+1)^2}{(\Lc_q+1)^2+\varkappa^2} q \bigr\rangle =0.
$$

As we will explain below, the quantity 
\begin{align*}
\bigl\langle q, \tfrac{(\Lc_q+1)^2}{(\Lc_q+1)^2+\varkappa^2} q \bigr\rangle
\end{align*}
is conserved by the $H_\ka$ flow. Consequently,
$$
\lim_{\varkappa\to\infty}\ \sup_{q\in Q_\ast}\ \bigl\langle q, \tfrac{(\Lc_q+1)^2}{(\Lc_q+1)^2+\varkappa^2} q \bigr\rangle =
\lim_{\varkappa\to\infty}\ \sup_{q\in Q}\ \bigl\langle q, \tfrac{(\Lc_q+1)^2}{(\Lc_q+1)^2+\varkappa^2} q \bigr\rangle =0,
$$
where the last equality follows from Theorem~\ref{T:equi-c} and the equicontinuity of $Q$.

That quantities of the form $\langle q, F(\Lc_q) q\rangle$ with $F\in L^\infty$ are conserved by the $H_\ka$ flow follows readily from the Lax pair representation for this flow.  Indeed, combining \eqref{beta Lax pair}, \eqref{eqn:Peter_op_ka}, \eqref{one true peter}, \eqref{En Lax pair}, \eqref{eqn:Peter_op_n}, \eqref{one true peter for En}, \eqref{eqn:H_ka T}, and \eqref{11:51}, we find that
\begin{equation}
\text{$q$ solves \eqref{IV Hk}} \qquad \iff \qquad \dd_t \Lc_q = \bigl[\mathcal{P_\ka},\Lc_q\bigr] \qquad \iff\qquad \dd_t q=\mathcal{P_\ka}q \label{Hk Lax pair}
\end{equation}
where, on the circle, $\mathcal P_\ka$ takes the form
\begin{align*}
\mathcal{P_\ka}&:=(\ka^3\mp \tfrac3{2\pi}\ka^2\E_0)\bigl[-i(\Lc_q+\ka)^{-1}\pm im_\ka C_+ \overline m_\ka \mp\tfrac{i}{2\pi}\beta_\ka(\Lc_q+\ka)^{-1} \mp\tfrac{i}{2\pi}\|m_\ka\|_2^2\bigr]\\
&\quad+i\ka^2 -\ka \partial_x  \mp\tfrac{2i}{\pi}\ka \E_0 \pm \tfrac{3i}{2\pi}\ka^2\beta -\tfrac{3i}{2\pi^2}\E_0^2\pm \tfrac{i}{2\pi}\ka \| \Lc_q m\|_{L^2}^2\pm \tfrac{i}{2\pi}\ka\E_0\\
&\quad\mp\tfrac{1}{2\pi}\ka^2\E_0\Bigl\{2i (\Lc_q+\ka)^{-1} \mp2i m C_+\overline{m}\pm\tfrac{i}{\pi} \beta(\Lc_q+\ka)^{-1}\pm\tfrac{i}{\pi} \|m\|^2_{L^2}\nonumber\\
&\quad-i \ka(\Lc_q+\ka)^{-2}\pm i\ka[(\Lc_q+\ka)^{-1}m] C_+\overline{m}\pm i\ka m C_+\overline{(\Lc_q+\ka)^{-1}m}\nonumber\\
&\quad\mp\tfrac{i}{2\pi}\ka\beta(\Lc_q+\ka)^{-2}\mp\tfrac{i}{2\pi}\ka \|m\|_{L^2}^2(\Lc_q+\ka)^{-1}\mp\tfrac{i}{\pi}\ka \lan m,(\Lc_q+\ka)^{-1}m\ran\Bigr\}.
\end{align*}
Thus, under the $H_\ka$ flow we have
\begin{align*}
\tfrac{d}{dt} \langle q, F(\Lc_q) q\rangle = \langle \mathcal{P_\ka} q, F(\Lc_q) q\rangle +  \langle q,\bigl[ \mathcal{P_\ka}, F(\Lc_q) \bigr] q\rangle + \langle q, F(\Lc_q)  \mathcal{P_\ka} q\rangle=0,
\end{align*}
where in the last equality we used the anti-selfadjointness of $\mathcal{P_\ka}$.
\end{proof}

\section{Global well-posedness}\label{sec:convergence}
In this section we present the proof of Theorem~\ref{T:GWP}. The key ingredient is the convergence result \eqref{H diff}.

\begin{proposition}\label{prop:diff_estimate}
Fix $M<M_\ast$. On $B_M$ we have $\na_\om H(q)\in H^{-5}$; in fact,
\begin{align}\label{H Lip}
\bigl\|\na_\om [H(q)-H(p)]\bigr\|_{H^{-5}}\lesssim_M \|q-p\|_{L^2}
\end{align} 
for all $q,p\in B_M$.  Moreover, for any $L^2$-equicontinuous set $Q\subset B_M$, we have
\begin{align}\label{H diff}
\lim_{\ka\to\infty}\,\sup_{q\in Q}\,\bigl\|\na_\om[H_\ka(q)-H(q)]\bigr\|_{H^{-5}}=0.
\end{align}
\end{proposition}

\begin{proof}
Throughout the proof, we use the shorthand $\beta$ and $m$ for $\beta(\ka,q)$ and $m(\ka,q)$, respectively.  We will also frequently use the following bound from \cite[Lemma~2.2]{killip2024sharp}:
\begin{equation}\label{eqn:H_r_bound}
\|C_+f\cdot C_+g\|_{H^{2r-1}}\lesssim\|f\|_{H^r}\|g\|_{H^r},
\end{equation}
valid for any $r<0$. Employing the embedding $L^1\hookrightarrow H^{-1}$, we may thus bound
\begin{align}\label{eqn:H_r_bound'}
\|C_+f\cdot C_+(\bar gh)\|_{H^{-3}}\lesssim \|f\|_{H^{-1}} \|\bar gh\|_{H^{-1}} \lesssim \|f\|_{L^2} \|g\|_{L^2} \|h\|_{L^2}.
\end{align}

Claim \eqref{H Lip} follows easily from \eqref{5:09} and \eqref{eqn:H_r_bound}.  Indeed, the key observation is
$$
\|C_+f\cdot C_+(\bar g h)'\|_{H^{-5}}\lesssim \|f\|_{H^{-2}} \|\bar gh\|_{H^{-1}}  \lesssim \|f\|_{L^2} \|g\|_{L^2} \|h\|_{L^2}.
$$

To prove \eqref{H diff}, we will apply \eqref{eqn:H_r_bound'} when some of the functions $f,g,h$ are $\ka m$ or $\ka m -q = - \Lc_q m$.  To control such contributions, we recall that by \eqref{E:equiv} we have
\begin{align}\label{1}
\|\ka m\|_{L^2} \lesssim \|q\|_{L^2}, 
\end{align}
while by \eqref{item2} we have
\begin{align}\label{2}
\|\ka m-q\|_{L^2} =\|\Lc_q m\|_{L^2} \to 0  \quad \text{as $\ka \to \infty$,}\quad\text{uniformly for $q\in Q$}.
\end{align}
Recall also that from \eqref{item4} we have
\begin{align}\label{3}
\|\ka\Lc_q (\Lc_q+\ka)^{-1}m\|_{L^2} \to 0  \quad \text{as $\ka \to \infty$,}\quad\text{uniformly for $q\in Q$}.
\end{align}

Once again we present the details in the case of the circle, which is the more challenging of the two geometries. Using \eqref{Hamiltonians on T} and \eqref{eqn:H_ka T}, we may write
$$
H_\ka(q)-H(q)= \tfrac12[\E_2^\ka(q)-\E_2(q)] \pm \tfrac3{4\pi} \E_0(q)[\E_1^\ka(q)-\E_1(q)]\mp\tfrac1{4\pi}\ka \E_0(q) \|\Lc_q m\|_{L^2}^2.
$$

We start by showing that the contribution of $\E_0(q)\na_\om[\E_1^\ka(q)-\E_1(q)]$ to the left-hand side of \eqref{H diff} is acceptable. Combining \eqref{eqn:E0 and E1} and \eqref{symplectic deriv reg E}, we have
\begin{align*}
\na_\om[\E_1^\ka(q)-\E_1(q)]&=-2i\ka q+2i\ka^2 m\mp2i\ka^2 m C_+(q\overline{m})\pm\tfrac{i}{\pi}\ka^2m\beta\pm\tfrac{i}{\pi}\ka^2 q\|m\|_{L^2}^2\\
&\quad+2q'\mp\tfrac{2i}{\pi}q\E_0(q).
\end{align*}
Using $(\Lc_q+\ka)m=q$, we write
\begin{align*}
-2i\ka q&+2i\ka^2 m\mp2i\ka^2 m C_+(q\overline{m})+ 2q' \\
&= -2i\ka \Lc_qm \mp2i\ka^2 m C_+(q\overline{m})+ 2( \Lc_qm)' + 2i\ka [ \Lc_qm\pm qC_+(\bar q m) ] \\
&= 2( \Lc_qm)'\pm 2i\ka  \Lc_qm \cdot C_+(q\overline{m}) \pm 2i\ka q C_+(\bar q m - q\overline{m})\\
&= 2( \Lc_qm)'\pm 2i\ka  \Lc_qm \cdot C_+(q\overline{m}) \pm 2i\ka q C_+(m\overline{\Lc_qm} - \overline{m}\Lc_qm).
\end{align*}
Thus, by \eqref{eqn:H_r_bound'}, \eqref{1}, and \eqref{2} we have
\begin{align}\label{E1 diff 1}
\bigl\|-2i\ka q+2i\ka^2 m&\mp2i\ka^2 m C_+(q\overline{m})+ 2q'\bigr\|_{H^{-3}}\nonumber\\
&\lesssim \|\Lc_q m\|_{L^2} + \|q\|_{L^2} \|\ka m\|_{L^2} \|\Lc_q m\|_{L^2}\to 0  \quad \text{as $\ka \to \infty$},
\end{align}
uniformly for $q\in Q$.

For the remaining terms in $\na_\om[\E_1^\ka(q)-\E_1(q)]$, we write
\begin{align*}
\pm\tfrac{i}{\pi}&\ka^2m\beta\pm\tfrac{i}{\pi}\ka^2 q\|m\|_{L^2}^2\mp\tfrac{2i}{\pi}q\E_0(q)\\
&=\pm\tfrac{i}{\pi}\bigl[ \ka(\ka m -q) \beta + \ka q\langle q, m\rangle + \ka^2q\langle m, m\rangle -2q\langle q, q\rangle  \bigr]\\
&=\pm\tfrac{i}{\pi} \bigl[ (\ka m -q) \langle q, \ka m \rangle + q\langle q, \ka m -q\rangle + q\langle\ka m -q, \ka m\rangle + q\langle q, \ka m-q\rangle  \bigr]\\
&=\mp\tfrac{i}{\pi} \bigl[  \Lc_qm \langle q, \ka m \rangle + q\langle q, \Lc_qm\rangle + q\langle  \Lc_qm, \ka m\rangle + q\langle q,  \Lc_qm\rangle  \bigr],
\end{align*}
and so,  by \eqref{eqn:H_r_bound'}, \eqref{1}, and \eqref{2} we have
\begin{align}\label{E1 diff 2}
\bigl\|\pm\tfrac{i}{\pi}\ka^2m\beta\pm\tfrac{i}{\pi}&\ka^2 q\|m\|_{L^2}^2\mp\tfrac{2i}{\pi}q\E_0(q)\bigr\|_{L^2}\lesssim \|\Lc_q m\|_{L^2} \|q\|_{L^2}^2\to 0  \quad \text{as $\ka \to \infty$},
\end{align}
uniformly for $q\in Q$.

Combining \eqref{E1 diff 1} and \eqref{E1 diff 2}, we get
\begin{align*}
\lim_{\ka\to\infty}\,\sup_{q\in Q}\,\bigl\|\E_0(q)\na_\om[\E_1^\ka(q)-\E_1(q)]\bigr\|_{H^{-3}} =0.
\end{align*}

Next we consider the contribution of $\ka \E_0(q) \na_\om \|\Lc_q m\|_{L^2}^2$ to the left-hand side of \eqref{H diff}. Recalling \eqref{eqn:na_om_Lqm} and prudently reordering the terms, we may write 
\begin{align}
\ka \na_\om\| \Lc_qm&\|_{L^2}^2\notag\\
&=-2i\ka q+4i\ka^2 m-2i\ka^3(\Lc_q+\ka)^{-1}m\label{01}\\
&\quad\pm\tfrac{i}{\pi}\ka^2 m\beta\mp \tfrac{i}{\pi}\ka^3m\| m\|_{L^2}^2\label{02}\\
&\quad\pm\tfrac{i}{\pi}\ka^2 m\beta\mp \tfrac{i}{\pi}\ka^3\beta(\Lc_q+\ka)^{-1}m\label{03}\\
&\quad\pm\tfrac{2i}{\pi}\ka^2 q\|m\|^2_{L^2}\mp \tfrac{2i}{\pi}\ka^3 q\lan m,(\Lc_q+\ka)^{-1}m\ran\label{04}\\
&\quad\mp2i\ka^2m C_+(q\overline{m})\pm 2i\ka^3[(\Lc_q+\ka)^{-1}m] C_+(q\overline{m})\label{05}\\
&\quad\mp 2i\ka^2 m C_+(q\overline{m})\pm 2i\ka^3 mC_+(q\overline{(\Lc_q+\ka)^{-1}m}).\label{06}
\end{align}
Using $(\Lc_q+\ka)m=q$, we observe that
\begin{align*}
\eqref{01}&=-2i\ka \Lc_q^2 (\Lc_q+\ka)^{-1}m=[-i\partial_x \mp q C_+ \bar q]\bigl (-2i\ka \Lc_q (\Lc_q+\ka)^{-1}m \bigr),\\
\eqref{02}&= \pm\tfrac{i}{\pi}\ka m\lan \Lc_q m,\ka m\ran, \\
\eqref{03}&=\pm\tfrac{i}{\pi} [\ka\Lc_q(\Lc_q+\ka)^{-1}m] \langle q,\ka m\rangle,\\
\eqref{04}&=\pm\tfrac{2i}{\pi}q\lan \ka m,\ka \Lc_q(\Lc_q+\ka)^{-1}m\ran,\\
\eqref{05}&=\mp 2i[\ka\Lc_q(\Lc_q+\ka)^{-1}m]C_+(q\cdot \ka\overline{m}),\\
\eqref{06}&=\mp 2i\ka m C_+(q\cdot \overline{\ka\Lc_q(\Lc_q+\ka)^{-1}m}).
\end{align*}
Using Cauchy--Schwarz,  \eqref{eqn:H_r_bound'}, \eqref{1}, \eqref{2}, and \eqref{3}, we may thus bound
\begin{align*}
\bigl\|\ka \E_0(q) \na_\om \|\Lc_q m\|_{L^2}^2 \bigr\|_{H^{-3}}&\lesssim (1+\|q\|_{L^2}^2)\|\ka \Lc_q (\Lc_q+\ka)^{-1}m\|_{L^2}+ \|q\|_{L^2}^2\|\Lc_q m\|_{L^2}, 
\end{align*}
which converges to zero as $\ka\to\infty$, uniformly for $q\in Q$.

It remains to control the contribution of
\begin{align*}
&\tfrac12\na_\om [\E_2^\ka(q)-\E_2(q)] \pm \tfrac3{4\pi}[\E_1^\ka(q)-\E_1(q)]\na_\om \E_0(q) \mp\tfrac1{4\pi}\ka \|\Lc_q m\|_{L^2}^2\na_\om \E_0(q)
\end{align*}
to the left-hand side of \eqref{H diff}. Using \eqref{eqn:E0 and E1}, \eqref{eqn: E2}, and \eqref{symplectic deriv reg E}, this takes the form
\begin{align*}
&\tfrac12\na_\om [\E_2^\ka(q)-\E_2(q)] \pm \tfrac3{4\pi}[\E_1^\ka(q)-\E_1(q)]\na_\om \E_0(q) \mp\tfrac1{4\pi}\ka \|\Lc_q m\|_{L^2}^2\na_\om \E_0(q)\\
&=-i\ka^3m\pm i\ka^3m C_+(q\overline{m})\mp\tfrac{i}{2\pi} \ka^3 m \beta\mp\tfrac{i}{2\pi}\ka^3q\|m \|^2_{L^2} +i\ka^2 q  -\ka q'\pm\tfrac{i}{\pi}\ka q\mathcal{E}_0-iq''\\
&\quad \mp2qC_+(|q|^2)'\mp \tfrac1{2\pi}\mathcal{E}_0q'\mp\tfrac{3i}{2\pi}q\mathcal{E}_1 \mp \tfrac{3i}{2\pi}q [-\ka^2\beta+\ka \E_0-\E_1]\pm \tfrac{i}{2\pi}\ka q \|\Lc_q m\|_{L^2}^2.
\end{align*}

Employing the identity $(\Lc_q+\ka)m=q$, we may rewrite 
\begin{align*}
\mp\tfrac{i}{2\pi}\ka^3q\|m \|^2_{L^2} \pm \tfrac{i}{2\pi}\ka q \|\Lc_q m\|_{L^2}^2 = \pm\tfrac{i}{2\pi}\ka q\E_0(q) \mp \tfrac{i}{\pi}\ka^2 q\beta
\end{align*}
and
\begin{align*}
 \mp\tfrac{i}{2\pi} \ka^3  \beta m \pm \tfrac{i}{2\pi} \ka^2  q\beta &=\pm \tfrac{i}{2\pi} \ka^2 \beta\Lc_qm = \pm \tfrac{i}{2\pi}\ka \E_0(q) \Lc_qm\mp \tfrac{i}{2\pi}\ka (\Lc_qm)\langle q, \Lc_qm\rangle. 
\end{align*}
Inserting this into the expression above and grouping terms carefully, we arrive at
\begin{align}
\tfrac12\na_\om [\E_2^\ka(q)&-\E_2(q)] \pm \tfrac3{4\pi}[\E_1^\ka(q)-\E_1(q)]\na_\om \E_0(q) \mp\tfrac1{4\pi}\ka \|\Lc_q m\|_{L^2}^2\na_\om \E_0(q)\notag\\
&=-i\ka^3m +i\ka^2 q -\ka q' -iq''  \label{07}\\
&\quad \pm i\ka^3m C_+(q\overline{m})\mp2qC_+(|q|^2)'-\tfrac{i}{2\pi}\ka^2q\beta C_+(\bar q m)\label{08}\\
&\quad\mp \tfrac1{2\pi}\mathcal{E}_0(q)q' \pm \tfrac{i}{2\pi}\ka \E_0(q) \Lc_qm +\tfrac{i}{2\pi}\ka^2q\beta C_+(\bar q m)\label{09}\\
&\quad\mp \tfrac{i}{2\pi}\ka (\Lc_qm)\langle q, \Lc_qm\rangle.\label{10}
\end{align}

Using repeatedly the identity $(\Lc_q+\ka)m=q$, we may rewrite 
\begin{align*}
\eqref{07}&=-i(\Lc_qm)'' \pm \ka \bigl[q C_+(\bar{q}m)\bigr]'\mp i\ka^2 q C_+(\bar{q}m).
\end{align*}
In view of \eqref{2}, the contribution of the first term above is acceptable:
\begin{align*}
\lim_{\ka\to\infty}\,\sup_{q\in Q}\,\bigl\|(\Lc_qm)''\bigr\|_{H^{-2}} =0.
\end{align*}
Grouping the other two terms with \eqref{08}, we have
\begin{align}
\pm \ka \bigl[q  C_+(\bar{q}m)&\bigr]' \mp i\ka^2 q C_+(\bar{q}m) +\eqref{08} \notag\\
&= \pm \bigl[q C_+(|q|^2)\bigr]' \mp\bigl[q C_+(\bar q\Lc_qm)\bigr]'\mp  i\ka^2(\Lc_q m) C_+(q\overline{m})\notag\\
&\quad \pm i\ka^2q C_+(q\overline m -\bar q m)\mp2qC_+(|q|^2)'-\tfrac{i}{2\pi}\ka^2q\beta C_+(\bar q m)\notag\\
&=\pm (\Lc_qm+\ka m)'C_+(|q|^2)\mp\bigl[q C_+(\bar q \Lc_qm)\bigr]'\mp  i\ka(\Lc_q m)C_+(|q|^2)\notag\\
&\quad \pm i\ka (\Lc_qm)C_+(q\overline{\Lc_qm})\pm i\ka^2 q C_+[\overline m \Lc_qm - m \overline{\Lc_qm}]\mp qC_+(|q|^2)'\notag\\
&\quad-\tfrac{i}{2\pi}\ka^2q\beta C_+(\bar q m)\notag\\
&=\pm (\Lc_qm)'C_+(|q|^2)\mp\bigl[q C_+(\bar q \Lc_qm)\bigr]'+  i\ka q C_+(\bar q m) C_+(|q|^2)\notag\\
&\quad  \pm \ka m' C_+(q\overline{\Lc_qm})-i\ka q C_+(\bar q m) C_+(q\overline{\Lc_qm})\pm q C_+\bigl(|\ka m|^2-|q|^2\bigr)' \notag\\
&\quad-i\ka^2qC_+\bigl[q\overline{m} C_+(\bar q m) - \bar q m C_-(q\overline m)\bigr]-\tfrac{i}{2\pi}\ka^2q\beta C_+(\bar q m)\notag\\
&=\pm (\Lc_qm)'C_+(|q|^2)\mp\bigl[q C_+(\bar q \Lc_qm)\bigr]'\pm \ka m' C_+(q\overline{\Lc_qm})\label{1:07}\\
&\quad  +  i\ka q C_+(\bar q m) C_+(q\ka\overline{m})\pm q C_+\bigl(|\ka m|^2-|q|^2\bigr)' \notag\\
&\quad -i\ka^2qC_+\bigl[q\overline{m} C_+(\bar q m) - \bar q m C_-(q\overline m)\bigr]-\tfrac{i}{2\pi}\ka^2q\beta C_+(\bar q m).\notag
\end{align}
Using \eqref{Cid} and support considerations, we may write
\begin{align*}
&-i\ka^2qC_+\bigl[q\overline{m} C_+(\bar q m) - \bar q m C_-(q\overline m)\bigr]\\
&=-i\ka^2qC_+\bigl[C_+(q\overline{m}) C_+(\bar q m) -\tfrac1{2\pi}\beta C_+(\bar q m) - C_-(\bar q m) C_-(q\overline m)+\tfrac1{2\pi}\beta C_-(q\overline m)\bigr]\\
&= -i\ka^2qC_+(q\overline{m}) C_+(\bar q m) +\tfrac{i}{2\pi}\ka^2q\beta C_+(\bar q m) + i\ka^2q\tfrac1{4\pi^2}\beta^2 - i\ka^2q\tfrac1{4\pi^2}\beta^2\\
&= -i\ka^2qC_+(q\overline{m}) C_+(\bar q m) +\tfrac{i}{2\pi}\ka^2q\beta C_+(\bar q m).
\end{align*}
Inserting this into \eqref{1:07}, we obtain
\begin{align*}
\pm &\ka \bigl[q  C_+(\bar{q}m)\bigr]' \mp i\ka^2 q C_+(\bar{q}m) +\eqref{08}\\
&=\pm (\Lc_qm)'C_+(|q|^2)\mp\bigl[q C_+(\bar q \Lc_qm)\bigr]'\pm \ka m' C_+(q\overline{\Lc_qm})\pm q C_+\bigl(|\ka m|^2-|q|^2\bigr)' .
\end{align*}
Using \eqref{eqn:H_r_bound}, \eqref{eqn:H_r_bound'}, and \eqref{1}, we may satisfactorily bound all these summands:
\begin{align*}
\bigl\| (\Lc_qm)'C_+(|q|^2) \bigr\|_{H^{-3}}&\lesssim \|\Lc_qm\|_{L^2}\|q\|_{L^2}^2,\\
\bigl\| \bigl[q C_+(\bar q \Lc_qm)\bigr]'\bigr\|_{H^{-4}}&\lesssim \|\Lc_qm\|_{L^2}\|q\|_{L^2}^2,\\
\bigl\| \ka m' C_+(q\overline{\Lc_qm})\bigr\|_{H^{-3}}&\lesssim \|\ka m'\|_{H^{-1}} \|\Lc_qm\|_{L^2}\|q\|_{L^2}\lesssim \|\Lc_qm\|_{L^2}\|q\|_{L^2}^2,\\
\bigl\|q C_+\bigl(|\ka m|^2-|q|^2\bigr)' \bigr\|_{H^{-5}}&\lesssim\|q\|_{H^{-2}}\bigl\||\ka m|^2-|q|^2\bigr\|_{H^{-1}}\lesssim  \|\Lc_q m\|_{L^2}\|q\|_{L^2}^2.
\end{align*}
By \eqref{2}, all these terms converge to zero as $\ka \to \infty$, uniformly for $q\in Q$.  In this way, we deduce that
\begin{align*}
\lim_{\ka\to\infty}\,\sup_{q\in Q}\,\bigl\|\eqref{07}+\eqref{08} \bigr\|_{H^{-5}}=0.
\end{align*}

We now turn to \eqref{09}. Observing that 
\begin{align*}
q'&=(\Lc_qm)' +i\ka\Lc_qm \pm i\ka qC_+(\bar{q}m),
\end{align*}
we may write
\begin{align*}
\eqref{09}= \mp\tfrac{1}{2\pi}\E_0(q)(\Lc_qm)'-\tfrac i{2\pi}\ka q C_+(\bar q m)\langle q, \Lc_qm\rangle.
\end{align*}
By \eqref{eqn:H_r_bound'}, \eqref{1}, and \eqref{2}, we obtain
\begin{align*}
\lim_{\ka\to\infty}\,\sup_{q\in Q}\,\bigl\|\eqref{09} \bigr\|_{H^{-3}}=0.
\end{align*}

Finally, we turn to \eqref{10}.  Recalling that $\Lc_qm=-im'\mp qC_+(\bar q m)$ and using Cauchy--Schwarz and \eqref{eqn:H_r_bound'}, we may bound
\begin{align*}
\bigl\|\ka (\Lc_qm)\langle q, \Lc_q m\rangle \bigr\|_{H^{-3}} \lesssim \|\ka m\|_{L^2} \bigl(1+ \|q\|_{L^2}^2\bigr) \|q\|_{L^2}\| \Lc_q m\|_{L^2}.
\end{align*}
By \eqref{1} and \eqref{2}, we thus have
\begin{align*}
\lim_{\ka\to\infty}\,\sup_{q\in Q}\,\bigl\|\eqref{10} \bigr\|_{H^{-3}}=0,
\end{align*}
which completes the verification of \eqref{H diff}.
\end{proof} 

\begin{proof}[Proof of Theorem \ref{T:GWP}]
Fix $M< M_\ast$ and an $L^2$-equicontinuous set $Q\subset B_M$. Our first goal is to show that for each $T>0$, we have
\begin{align}\label{1:39}
\lim_{\ka, \varkappa\to \infty} \, \sup_{q_0\in Q} \, \bigl\| e^{t\na_\om H_{\ka}}q_0 - e^{t\na_\om H_{\varkappa}}q_0  \bigr\|_{C_t([-T,T], L^2_+)} = 0,
\end{align}
which is to say that the $H_\ka$ flows are Cauchy in $C_tL_+^2$ as $\ka \to \infty$, uniformly on $Q$.

Recall that by Theorem~\ref{thm:GWP_ka}, the flows $e^{t\na_\om H_\ka}q_0$ with $\ka\geq1$ are globally defined and the set
$$
Q_{\ast}:=\{e^{s\na_\om H_{\ka}}e^{t\na_\om H_{\vk}}q_0: q_0\in Q,\,\ka,\vk\geq 1,\,s,t\in\R\}
$$
is $L^2$-equicontinuous.

We begin by proving an analogue of \eqref{1:39} in the $C_tH^{-5}$ topology.  Using the commutativity of the $H_\ka$ flows for different values of $\ka$ (cf. Proposition~\ref{P:beta comm}) and the fundamental theorem of calculus, we may bound
\begin{align*}
\sup_{q_0\in Q}\, \bigl\|e^{t \na_\om H_{\ka}}q_0 &-e^{t \na_\om H_{\vk}}q_0\bigr\|_{C_t([-T,T],H^{-5})}\\
&=\sup_{q_0\in Q}\, \sup_{t\in[-T,T]}\, \bigl\|(e^{t\na_\om(H_{\ka}-H_{\vk})}-1)\circ e^{t\na_\om H_{\vk}}q_0\bigr\|_{H^{-5}}\\
&\leq T\sup_{q\in Q_{\ast}}\|\na_\om(H_\ka-H_{\vk})(q)\|_{H^{-5}}.
\end{align*}
By Proposition~\ref{prop:diff_estimate},
\begin{align*}
\lim_{\ka, \varkappa\to \infty} \,\sup_{q\in Q_{\ast}}\|\na_\om(H_\ka-H_{\vk})(q)\|_{H^{-5}}=0,
\end{align*}
which yields
\begin{align}\label{1:40}
\lim_{\ka, \varkappa\to \infty} \, \sup_{q_0\in Q} \, \bigl\| e^{t\na_\om H_{\ka}}q_0 - e^{t\na_\om H_{\varkappa}}q_0  \bigr\|_{C_t([-T,T], H^{-5})} = 0.
\end{align}

To upgrade \eqref{1:40} to \eqref{1:39}, we use the equicontinuity of $Q_\ast$.  Indeed, given $\eps>0$ we choose $N=N(\eps, Q_*)$ such that
$$
\sup_{q\in Q_\ast}\|P_{>N} q\|_{L^2} \leq \eps.
$$ 
Then for $\ka,\vk\geq1$, we may estimate
\begin{align*}
\sup_{q_0\in Q}\,\sup_{t\in[-T,T]}\, &\bigl\|e^{t\na_\om H_{\ka}}q_0 - e^{t\na_\om H_{\varkappa}}q_0\bigr\|_{L^2}\\
&\leq\sup_{q_0\in Q}\,  \sup_{t\in[-T,T]}\,\bigl\|P_{\leq N}\bigl[e^{t\na_\om H_{\ka}}q_0 - e^{t\na_\om H_{\varkappa}}q_0\bigr]\bigr\|_{L^2}\\
&\quad +\sup_{q_0\in Q}\,  \sup_{t\in[-T,T]}\,\bigl\|P_{> N}\bigl[e^{t\na_\om H_{\ka}}q_0 - e^{t\na_\om H_{\varkappa}}q_0\bigr]\bigr\|_{L^2}\\
&\lesssim N^5\sup_{q_0\in Q}\,  \sup_{t\in[-T,T]}\, \bigl\|e^{t\na_\om H_{\ka}}q_0 - e^{t\na_\om H_{\varkappa}}q_0\bigr\|_{H^{-5}}+\eps.
\end{align*}
Together with \eqref{1:40}, this proves \eqref{1:39}. In particular, for each $q_0\in Q$, this guarantees the existence of a trajectory $q(t,q_0)$ belonging to $C_tL_+^2$ so that 
\begin{align}\label{1:41}
\lim_{\ka\to \infty}\, \sup_{q_0\in Q} \, \bigl\| e^{t\na_\om H_{\ka}}q_0 - q(t,q_0)\bigr\|_{C_t([-T,T],L^2_+)}=0.
\end{align}

Next, we confirm that $q(t)=q(t,q_0)$ does indeed solve $\dd_t q=\na_\om H(q)$ with initial data $q_0$.  As a solution to \eqref{IV Hk}, $q_\kappa(t):=e^{t\na_\om H_{\ka}}q_0$ satisfies $q_\kappa(0)=q_0$ and
\begin{equation}\label{eqn:equality}
q_{\ka}(t)=q_0+\int_0^t \na_\om H_\ka (q_\ka(s))\,ds
\end{equation}
in $H^{-5}$ sense for each $\kappa\geq 1$ and $t\in\R$.  We now wish to send $\kappa\to\infty$.  We know already that  $q_{\ka}(t)\rightarrow q(t)$ in $L^2$ and so turn our attention to the integrand:
\begin{align*}
\sup_{s} \|&\na_\om H_\ka(q_\ka(s))-\na_\om H(q(s))\|_{H^{-5}}\\
&\leq\sup_{s}\|\na_\om (H_\ka-H)(q_\ka(s))\|_{H^{-5}}+\sup_{s}\|\na_\om H(q_\ka(s))-\na_\om H(q(s))\|_{H^{-5}}.
\end{align*}
Proposition \ref{prop:diff_estimate} ensures that both terms vanish as $\ka\to\infty$. In this way, we can take the limit in \eqref{eqn:equality} to obtain
\begin{equation}\label{1797}
q(t)=q_0+\int_0^t \na_\om H (q(s))\,ds
\end{equation}
in $C_tH^{-5}$ sense. In particular, $q$ is a distributional solution.  It is not difficult to upgrade this notion of solution; for example, one can see that  both \eqref{1797} and the analogous Duhamel formulation hold in $C_t^1H^{-5}$ sense.

It remains to prove continuity of the flow map. Suppose $q_n\rightarrow q_0$ in $L^2_+$.  The set $Q=\{q_n:n\geq1\}\cup\{q_0\}$ is compact in $L^2$ and so equicontinuous. Thus by \eqref{1:41}, for each $\eps>0$ and $T>0$, there exists $\ka\geq1$ so that
$$
\sup_{u\in Q}\,\bigl\|(e^{t\na_\om H}-e^{t\na_\om H_\ka})u\bigr\|_{C_t([-T,T],L^2_+)}\leq\eps.
$$
For this same $\ka$, well-posedness of the $H_\ka$ flow ensures that 
$$
\bigl\|e^{t\na_\om H_\ka}q_n-e^{t\na_\om H_\ka}q_0\bigr\|_{C_t([-T,T],L^2_+)}\leq\eps
$$
for all $n$ large enough. For such $n$, we then have
\begin{align*}
&\bigl\|e^{t\na_\om H}q_n-e^{t\na_\om H}q_0\bigr\|_{C_t([-T,T],L^2_+)}\\
&\leq 2\sup_{u\in Q}\, \bigl\|(e^{t\na_\om H}-e^{t\na_\om H_\ka})u\bigr\|_{C_t([-T,T],L^2_+)}+\bigl\|e^{t\na_\om H_\ka}q_n-e^{t\na_\om H_\ka}q_0\bigr\|_{C_t([-T,T],L^2_+)}\\
&\leq3\eps.
\end{align*}
This completes the proof. 
\end{proof}

\bibliography{refs2}

\end{document}